\documentclass{amsart}
\usepackage{amsmath}

\usepackage[all]{xy}

\newtheorem{theorem}{Theorem}[section]
\newtheorem{lemma}[theorem]{Lemma}

\theoremstyle{definition}
\newtheorem{definition}[theorem]{Definition}
\newtheorem{example}[theorem]{Example}
\newtheorem{proposition}[theorem]{Proposition}
\newtheorem{corollary}[theorem]{Corollary}

\theoremstyle{remark}
\newtheorem{remark}[theorem]{Remark}

\numberwithin{equation}{section}

\begin{document}

\title{On a spectral sequence for twisted cohomologies}

\author{Weiping Li}
\address{Department of Mathematics,
         Oklahoma State University,
         Stillwater, OK 74078, U.S.A}

\email{wli@math.okstate.edu}

\author{Xiugui Liu}
\address{School of Mathematical Sciences and LPMC,
         Nankai University,
         Tianjin 300071, P.R.China}
\email{xgliu@nankai.edu.cn}
\thanks{The second author was partially
supported by NCET and NNSFC (No. 10771105). }

\author{He Wang}
\address{School of Mathematical Sciences,
         Nankai University,
         Tianjin 300071, P.R.China}
\email{wanghe85@yahoo.com.cn}

\subjclass[2000]{Primary 58J52; Secondary 55T99, 81T30}



\keywords{Spectral sequence, twisted de Rham cohomology, Massey
product, differential}

\begin{abstract}
Let ($\Omega^{\ast}(M), d$) be the de Rham cochain complex for a
smooth compact closed manifolds $M$ of dimension $n$. For an
odd-degree closed form $H$, there are a twisted de Rham cochain
complex $(\Omega^{\ast}(M), d+H_\wedge)$ and its associated twisted
de Rham cohomology $H^*(M,H)$. We show that there exists a spectral
sequence $\{E^{p, q}_r, d_r\}$ derived from the filtration
$F_p(\Omega^{\ast}(M))=\bigoplus_{i\geq p}\Omega^i(M)$ of
$\Omega^{\ast}(M)$, which converges to the twisted de Rham
cohomology $H^*(M,H)$. We also show that the differentials in the
spectral sequence can be given in terms of cup products and specific
elements of Massey products as well, which generalizes a result of
Atiyah and Segal. Some results about the indeterminacy of
differentials are also given in this paper.

\end{abstract}

\maketitle

\section{Introduction}

Let $M$ be a smooth compact closed manifold of dimension $n$, and
$\Omega^{\ast}(M)$ the space of smooth differential forms over
$\mathbb{R}$ on $M$. We have the de Rham cochain complex
$(\Omega^{\ast}(M), d),$ where $
d:\Omega^p(M)\rightarrow\Omega^{p+1}(M)$ is the exterior
differentiation, and its cohomology $H^{\ast}(M)$ (the de Rham
cohomology). The de Rham cohomology with coefficients in a flat
vector bundle is an extension of the de Rham cohomology.

The twisted de Rham cohomology was first studied by Rohm and Witten
\cite{R-W} for the antisymmetric field in superstring theory. By
analyzing the massless fermion states in the string sector, Rohm and
Witten obtained the twisted de Rham cochain complex
$(\Omega^{\ast}(M), d+H_3)$ for a closed 3-form $H_3$, and mentioned
the possible generalization to a sum of odd closed forms. A key
feature in the twisted de Rham cohomology  is that the theory is not
integer graded but (like K-theory) is filtered with the grading mod
$2$. This has a close relation with the twisted K-theory and the
Atiyah-Hirzebruch spectral sequence (see \cite{A-S}).

Let $H$ be $\sum_{i=1}^{[\frac{n-1}{2}]}H_{2i+1}$, where $H_{2i+1}$
is a closed $(2i+1)$-form. Then one can define a new operator
$D=d+H$ on $\Omega^{\ast}(M)$, where $H$ is understood as an
operator acting by exterior multiplication (for any differential
form $w$, $H(w)=H\wedge w$). As in \cite{A-S, R-W}, there is a
filtration on $(\Omega^{\ast}(M),D)$:
\begin{equation}K_p=F_p(\Omega^{\ast}(M))=\bigoplus\limits_{i\geq
p}\Omega^i(M).\end{equation} This filtration gives rise to a
spectral sequence
\begin{equation}\label{ss}
\{E^{p, q}_r, d_r\}
\end{equation}converging to the twisted de Rham cohomology
$H^{\ast}(M, H)$ with
\begin{equation}E_2^{p,q}\cong\left\{\begin{array}{ll}H^p(M) &
~~~\mbox{$q$
is even,} \\
0& ~~~\mbox{$q$ is odd.}\end{array}\right.\end{equation}

For convenience, we first fix some notations in this paper. The
notation $[r]$ denotes the greatest integer part of $r\in {\mathbb
R}$. In the spectral sequence (\ref{ss}) $\{E^{p, q}_r, d_r\}$, for
any $[y_p]_k\in E_k^{p,q}$, $[y_p]_{k+l}$ represents its class to
which $[y_p]_k$ survives in $E_{k+l}^{p,q}.$ In particular, as in
Proposition~\ref{e3}, for $x_p\in E_1^{p,q}$, $[x_p]_2=[x_p]_3\in
E_2^{p,q}=E_3^{p,q}$ represents the de Rham cohomology class
$[x_p]$. $d_r[x_p]$ represents a class in $E_{2}^{p+r, q-r+1}$ which
survives to $d_{r}[x_p]_{r}\in E_{r}^{p+r, q-r+1}$.

In the appendix I of \cite{R-W}, Rohm and Witten first gave a
description of the differentials $d_3$ and $d_5$ for the case when
$D=d+H_3$. Atiyah and Segal \cite{A-S} showed a method about how to
construct the differentials in terms of Massey products, and gave a
generalization of Rohm and Witten's result: {the iterated Massey
products with $H_3$ give (up to sign) all the higher differentials
of the spectral sequence for the twisted cohomology} (see
\cite[Proposition 6.1]{A-S}). Mathai and Wu in \cite[p. 5]{M-W}
considered the general case that
$H=\sum_{i=1}^{[\frac{n-1}{2}]}H_{2i+1}$ and claimed, without proof,
that {$d_2=d_4=\cdots=0$, while $d_3$, $d_5$, $\cdots$ are given by
the cup products with $H_3$, $H_5$, $\cdots$ and by the higher
Massey products with them.} Motivated by the method in \cite{A-S},
we give an explicit description of the differentials in the spectral
sequence (\ref{ss}) in terms of Massey products.

We now describe our main results. Let $A$ denote a defining system
for the $n$-fold Massey product $\langle x_1,x_2,\cdots,x_n\rangle$
and $c(A)$ its related cocycle (see Definition \ref{5.00}). Then
\begin{equation}\langle x_1,x_2,\cdots,x_n\rangle=\{c(A)|A~{\rm
is~a~defining~system~for}\langle
x_1,x_2,\cdots,x_n\rangle\}\end{equation} by Definition
\ref{Massey2}. To obtain our desired theorems by specific elements
of Massey products, we restrict the allowable choices of defining
systems for Massey products (cf. \cite{sh}). By Theorems \ref{4.1}
and \ref{4.2} in this paper, there are defining systems for the two
Massey products we need (see Lemma \ref{ds}). The notation $\langle
\underbrace{H_3,\cdots,H_3}\limits_{t+1},x_p\rangle_{A}$ in Theorem
\ref{main2} below denotes a cohomology class in $H^{\ast}(M)$
represented by $c(A)$, where $A$ is a defining system obtained by
Theorem \ref{4.1} (see Definition \ref{Massey1}). Similarly, the
notation $\langle
\underbrace{H_{2s+1},\cdots,H_{2s+1}}\limits_{l},x_p\rangle_{A}$ in
Theorem \ref{main3} below denotes a cohomology class in
$H^{\ast}(M)$ represented by $c(A)$, where $A$ is a defining system
obtained by Theorem \ref{4.2} (see Definition \ref{Massey1}).

\begin{theorem}\label{main2}
For $H=\sum_{i=1}^{[\frac{n-1}{2}]}H_{2i+1}$ and $[x_p]_{2t+3}\in
E_{2t+3}^{p,q}$ $(t\geq 1)$, the differential of the spectral
sequence (\ref{ss}) $d_{2t+3}: E_{2t+3}^{p, q}\to E_{2t+3}^{p+2t+3,
q-2t-2}$ is given by
\begin{equation*}
d_{2t+3}[x_p]_{2t+3}=(-1)^t[\langle
\underbrace{H_3,\cdots,H_3}\limits_{t+1},x_p\rangle_{A}]_{2t+3},
\end{equation*}
and $[\langle
\underbrace{H_3,\cdots,H_3}\limits_{t+1},x_p\rangle_{A}]_{2t+3}$ is
independent of the choice of the defining system $A$ obtained from
Theorem \ref{4.1}.
\end{theorem}

Specializing Theorem \ref{main2} to the case in which $H=H_{2s+1}$
$(s\geq2)$, we obtain
\begin{equation}d_{2t+3}[x_p]_{2t+3}=(-1)^{t}[\langle
\underbrace{0,\cdots,0}\limits_{t+1},x_p\rangle_A]_{2t+3}.\end{equation}
Obviously, much information has been concealed in the expression
above. In particular, we give a more explicit expression of
differentials for this special case which is compatible with Theorem
\ref{main2} (see Remark \ref{5.9}).

\begin{theorem}\label{main3} For  $H=H_{2s+1}$ $(s\geq 1)$ only and $[x_p]_{2t+3}\in
E_{2t+3}^{p,q}$ $(t\geq 1)$, the differential of the spectral
sequence (\ref{ss}) $d_{2t+3}: E_{2t+3}^{p, q}\to E_{2t+3}^{p+2t+3,
q-2t-2}$ is given by
\[d_{2t+3}[x_p]_{2t+3}=\left\{ \begin{array}{ll}
[H_{2s+1} \wedge x_p]_{2t+3} & t=s-1,\\

(-1)^{l-1}[\langle \underbrace{H_{2s+1},\cdots,H_{2s+1}}\limits_{l},x_p\rangle_{B}]_{2t+3}  & t=ls-1~(l\geq 2),\\

0 & \text{otherwise,} \end{array} \right.
\]
and $[\langle
\underbrace{H_{2s+1},\cdots,H_{2s+1}}\limits_{l},x_p\rangle_{B}]_{2t+3}$
is independent of the choice of the defining system $B$ obtained
from Theorem \ref{4.2}.
\end{theorem}

Atiyah and Segal in \cite{A-S} gave the differential expression in
terms of Massey products when $H=H_3$ (see \cite[Proposition
6.1]{A-S}).
Obviously, the result of Atiyah and Segal is a special case of
Theorem \ref{main3}.

Theorem \ref{main2} is essentially Theorem \ref{differentials1}, and
Theorem \ref{main3} is Theorem \ref{differentials2}. Some of the
results above are known to experts in this field, but there is a
lack of mathematical proof in the literature.

This paper is organized as follows. In Section 2, we recall some
backgrounds about the twisted de Rham cohomology. In Section 3, we
consider the structure of the spectral sequence converging to the
twisted de Rham cohomology, and give the differentials $d_i$ ($1\leq
i\leq 3$) and $d_{2k}$ $(k\geq 1)$. With the formulas of the
differentials in $E_{2t+3}^{p,q}$ in Section 4, Theorems \ref{main2}
and \ref{main3} (i.e., Theorems~\ref{differentials1} and
\ref{differentials2}) are shown in Section 5. In Section 6, we
discuss the indeterminacy of differentials of the spectral sequence
(\ref{ss}).

\section{Twisted de Rham cohomology }

For completeness, in this section we recall some knowledge about the
twisted de Rham cohomology. Let $M$ be a smooth compact closed
manifold of dimension $n$, and $\Omega^{\ast}(M)$ the space of
smooth differential forms on $M$. We have the de Rham cochain
complex $(\Omega^{\ast}(M), d)$ with the exterior differentiation
$d:\Omega^p(M)\rightarrow\Omega^{p+1}(M)$, and its cohomology
$H^{\ast}(M)$ (the de Rham cohomology).

Let $H$ denote $\sum_{i=1}^{[\frac{n-1}{2}]}H_{2i+1}$, where
$H_{2i+1}$ is a closed $(2i+1)$-form. Define a new operator $D=d+H$
on $\Omega^{\ast}(M)$, where $H$ is understood as an operator acting
by exterior multiplication (for any differential form $w$,
$H(w)=H\wedge w$, also denoted by $H_{\wedge}$). It is easy to show
that
$$D^2=(d+H)^2=d^2+dH+Hd+H^2=0.$$ However $D$ is not homogeneous on the space of smooth differential
forms $\Omega^{\ast}(M)=\bigoplus\limits_{i\geq0}\Omega^i(M)$.

Define $\Omega^{\ast}(M)$ a new (mod 2) grading
\begin{equation}
\Omega^{\ast}(M)=\Omega^o(M)\oplus\Omega^{e}(M),
\end{equation}
 where
\begin{equation}
\begin{array}{lclc}
\Omega^o(M)=\bigoplus\limits_{i\geq0\atop \rm{i\equiv 1 \pmod
2}}\Omega^i(M) \quad\mbox{and}\quad
\Omega^e(M)=\bigoplus\limits_{i\geq0\atop \rm{i\equiv 0\pmod
2}}\Omega^i(M).
\end{array}
\end{equation}
Then $D$ is homogenous for this new (mod 2) grading:
$$\Omega^e(M)\stackrel{D}\longrightarrow \Omega^o(M)\stackrel{D}\longrightarrow \Omega^e(M).$$
 Define the
twisted de Rham cohomology groups of $M$:
\begin{equation}\label{x}
H^o(M,H)=\frac{\mathrm{ker}[D:\Omega^o(M)\rightarrow\Omega^e(M)]}{\mathrm{im}[D:\Omega^e(M)\rightarrow\Omega^o(M)]}
\end{equation}
and
\begin{equation} \label{2.1}
H^e(M,H)=\frac{\mathrm{ker}[D:\Omega^e(M)\rightarrow\Omega^o(M)]}{\mathrm{im}[D:\Omega^o(M)\rightarrow\Omega^e(M)]}.
\end{equation}

\begin{remark}
${\rm (i)}$ The twisted de Rham cohomology groups $H^*(M, H)$ $(*=o,
e)$ depend on the closed form $H$ and not just on its cohomology
class. If $H$ and $H^{'}$ are cohomologous, then $H^*(M, H)\cong
H^*(M, H^{'})$  $($see \cite[\S 6]{A-S}$)$.

${\rm (ii)}$ The twisted de Rham cohomology is also an important
homotopy invariant $($see \cite[\S 1.4]{M-W}$)$.
\end{remark}

Let $E$ be a flat vector bundle over $M$ and ${\Omega}^i(M, E)$ be
the space of smooth differential $i$-forms on $M$ with values in
$E$. A flat connection on $E$ gives a linear map
\[\nabla^E: {\Omega}^i(M, E) \to {\Omega}^{i+1}(M, E)\]
such that, for any smooth function $f$ on $M$ and any $\omega \in
{\Omega}^i(M, E)$,
\[\nabla^E (f \omega ) = df \wedge \omega + f \cdot \nabla^E\omega, \ \ \ \nabla^E \circ \nabla^E = 0.\]
Similarly, define $\Omega^{\ast}(M, E)$ a new $\pmod 2$ grading
\begin{equation}
\Omega^{\ast}(M, E)=\Omega^o(M, E)\oplus\Omega^{e}(M, E),
\end{equation}
 where
\begin{equation}
\begin{array}{lclc}
\Omega^o(M, E)=\bigoplus\limits_{i\geq0\atop \rm{i\equiv 1 \pmod
2}}\Omega^i(M, E) \quad\mbox{and}\quad \Omega^e(M,
E)=\bigoplus\limits_{i\geq0\atop \rm{i\equiv 0\pmod 2}}\Omega^i(M,
E).
\end{array}
\end{equation}
Then $D^E = \nabla^E + H_{\wedge}$ is homogenous for the new (mod 2)
grading:
$$\Omega^e(M, E)\stackrel{D^E}\longrightarrow \Omega^o(M, E)\stackrel{D^E}\longrightarrow \Omega^e(M, E).$$
 Define the
twisted de Rham cohomology groups of $E$:
\begin{equation}\label{y}
H^o(M,E, H)=\frac{\mathrm{ker}[D^E:\Omega^o(M,
E)\rightarrow\Omega^e(M, E)]}{\mathrm{im}[D^E:\Omega^e(M,
E)\rightarrow\Omega^o(M, E)]}
\end{equation}
and
\begin{equation} \label{2.2}
H^e(M,E, H)=\frac{\mathrm{ker}[D^E:\Omega^e(M,
E)\rightarrow\Omega^o(M, E)]}{\mathrm{im}[D^E:\Omega^o(M,
E)\rightarrow\Omega^e(M, E)]}.
\end{equation}
Results proved in this paper are also true for  the twisted de Rham
cohomology groups $H^*(M, E, H)$ ($*=o, e$) with twisted
coefficients in $E$ without any change.

\section{ A spectral sequence for twisted de Rham cohomology and its differentials $d_i$ ($1\leq i\leq 3$), $d_{2k}$ ($k\geq 1$)}

Recall that $D=d+H$ and $H=\sum_{i=1}^{[\frac{n-1}{2}]}H_{2i+1}$,
where $H_{2i+1}$ is a closed $(2i+1)$-form. Define the usual
filtration on the graded vector space $\Omega^{\ast}(M)$ to be
\begin{displaymath}
K_p = F_p(\Omega^{\ast}(M))=\bigoplus\limits_{i\geq p}\Omega^i(M),
\end{displaymath}
and $K= K_0 = \Omega^{\ast}(M)$. The filtration is bounded and
complete,
\begin{equation} \label{3.1}
K\equiv K_0\supset K_1\supset K_2\supset\cdots\supset K_n\supset
K_{n+1}=\{0\}.
\end{equation}
We have $D(K_p)\subset K_p$ and $D(K_p)\subset K_{p+1}$. The
differential $D(=d+H)$ does not preserve the grading of the de Rham
complex. However, it does preserve the filtration $\{K_p\}_{p\geq
0}.$

The filtration $\{K_p\}_{p\geq 0}$ gives an exact couple (with
bidegree) (see \cite{M}). For each $p$, $K_p$ is a graded vector
space with
\begin{displaymath}
K_p=(K_p\cap\Omega^o(M))\oplus(K_p\cap\Omega^e(M))=K_p^o\oplus
K_p^e,
\end{displaymath}
where $K_p^o=K_p\cap\Omega^o(M)$ and $K_p^e=K_p\cap\Omega^e(M)$. The
cochain complex $(K_p,D)$ is induced by
$D:\Omega^{\ast}(M)\longrightarrow\Omega^{\ast}(M)$. Similar to
(\ref{2.1}), there are two well-defined
 cohomology groups $H_D^e(K_p)$ and $H_D^o(K_p)$.
Note that a cochain complex with grading
\begin{displaymath}
K_p/K_{p+1}=(K_p^o/K_{p+1}^o)\oplus(K_p^e/K_{p+1}^e)
\end{displaymath}
derives cohomology groups $H_D^o(K_p/K_{p+1})$ and
$H_D^e(K_p/K_{p+1})$. Since $D(K_p)\subset K_{p+1}$, we have $D=0$
in the cochain complex $(K_p/K_{p+1},D)$.

\begin{lemma}\label{3.8}
For the cochain complex $(K_p/K_{p+1},D)$, we have
\begin{displaymath}
H_D^o(K_p/K_{p+1})\cong\left\{\begin{array}{ll}\Omega^p(M) &
~~~\mbox{$p$
is odd,} \\
0& ~~~\mbox{$p$ is even.}\end{array}\right.
\end{displaymath} and
\begin{displaymath}
H_D^e(K_p/K_{p+1})\cong\left\{\begin{array}{ll}\Omega^p(M) &
~~~\mbox{$p$
is even,} \\
0& ~~~\mbox{$p$ is odd.}\end{array}\right.
\end{displaymath}
\end{lemma}

\begin{proof} If $p$ is odd, then
$$K_p\cap \Omega^{e}(M)=K_{p+1}\cap \Omega^e(M)~ {\rm and}  ~
(K_p\cap\Omega^{e}(M) )\left/(K_{p+1}\cap \Omega^e(M))\right.=0.$$
Also
$$(K_p\cap \Omega^{o}(M))\left/(K_{p+1}\cap \Omega^o(M))\right.=
K^o_p/K^o_{p+1}\cong \Omega^p(M), $$ and
$$H_D^o(K_p/K_{p+1})\cong\Omega^p(M) ~{\rm and}~ H_D^e(K_p/K_{p+1})=0.$$
Similarly for even $p$, we have
$$H_D^e(K_p/K_{p+1})\cong\Omega^p(M) ~{\rm and}~ H_D^o(K_p/K_{p+1})=0.$$
\end{proof}

By the filtration (\ref{3.1}), we obtain a short exact sequence of
cochain complexes
\begin{equation} \label{3.2}
0\longrightarrow K_{p+1}\stackrel{i}\longrightarrow
K_p\stackrel{j}\longrightarrow K_p/K_{p+1}\longrightarrow 0,
\end{equation}
which gives rise to a long exact sequence of cohomology groups
\begin{equation}\label{3.3}
\begin{array}{cc}
\cdots\longrightarrow &
H_D^{p+q}(K_{p+1})\stackrel{i^{\ast}}\longrightarrow
H_D^{p+q}(K_p)\stackrel{j^{\ast}}\longrightarrow
H_D^{p+q}(K_p/K_{p+1})\\

& \stackrel{\delta}\longrightarrow
H_D^{p+q+1}(K_{p+1})\stackrel{i^{\ast}}\longrightarrow
H_D^{p+q+1}(K_p)\stackrel{j^{\ast}}\longrightarrow\cdots.
\end{array}
\end{equation} Note that in the exact sequence above,

\begin{equation*}
H_D^i(K_p)=\left\{\begin{array}{ll} H_D^e(K_p)&\mbox{$i$ is even,}\\
H_D^o(K_p)&\mbox{$i$ is odd.}
\end{array}
\right.
\end{equation*}
and
\begin{equation*}
H_D^i(K_p/K_{p+1})=\left\{\begin{array}{ll} H_D^e(K_p/K_{p+1})& \mbox{$i$ is even,}\\
H_D^o(K_p/K_{p+1})&\mbox{$i$ is odd.}
\end{array}
\right.
\end{equation*}
Let
\begin{equation} \label{3.4}
\begin{array}{cc}
E_1^{p,q}=H_D^{p+q}(K_p/K_{p+1}),\quad D_1^{p,q}=H_D^{p+q}(K_p),\\

i_1=i^{\ast},\quad j_1=j^{\ast},\quad\mbox{and}\quad k_1=\delta.
\end{array}
\end{equation}
We get an exact couple from the long exact sequence (\ref{3.3})
\begin{equation} \label{3.5}
\xymatrix{ D_1^{\ast,\ast} \ar[rr]^{i_1}
&  &    D_1^{\ast,\ast} \ar[dl]^{j_1}\\
& E_1^{\ast,\ast}\ar[ul]^{k_1}}
\end{equation}
with $i_1$ of bidegree $(-1,1)$, $j_1$ of bidegree $(0,0)$ and $k_1$
of bidegree $(1,0)$.

We have $d_1=j_1k_1:E_1^{\ast,\ast}\longrightarrow E_1^{\ast,\ast}$
with bidegree $(1,0)$, and $d_1^2=j_1k_1j_1k_1=0$. By (\ref{3.5}),
we have the derived couple
\begin{equation} \label{3.6}
\xymatrix{ D_2^{\ast,\ast} \ar[rr]^{i_2}
&  &    D_2^{\ast,\ast} \ar[dl]^{j_2}    \\
& E_2^{\ast,\ast}\ar[lu]^{k_2}}
\end{equation}
by  the following:
\begin{enumerate}
\item $D_2^{\ast,\ast}=i_1D_1^{\ast,\ast}$,
$E_2^{\ast,\ast}=H_{d_1}(E_1^{\ast,\ast})$.

\item $i_2=i_1|_{D_2^{\ast,\ast}}$, also denoted by $i_1$.

\item If $a_2=i_1a_1\in D_2^{\ast,\ast}$, define
$j_2(a_2)=[j_1a_1]_{d_1}$, where $[\ \ ]_{d_1}$ denotes the
cohomology class in $H_{d_1}(E_1^{\ast,\ast}).$

\item For $[b]_{d_1}\in E_2^{\ast,\ast}=H_{d_1}(E_1^{\ast,\ast}),$
define $k_2([b]_{d_1})=k_1b\in D_2^{\ast,\ast}$.\\
\end{enumerate}
The derived couple (\ref{3.6}) is also an exact couple, and $j_2$
and $k_2$ are well-defined (see \cite{Ma, M}).

\begin{proposition} \label{specs}
 ${\rm (i)}$ There exists a spectral sequence $(E^{p,
q}_r, d_r)$ derived from the filtration $\{K_n\}_{n\geq 0}$, where
$E_1^{p, q} = H^{p+q}_D(K_p/K_{p+1})$ and $d_1=j_1k_1$, and $E^{p,
q}_2=H_{d_1}(E^{p, q}_1)$ and $d_2=j_2k_2.$ The bidegree of $d_r$ is
$(r, 1-r)$.

${\rm (ii)}$ The spectral sequence $\{E^{p, q}_r, d_r\}$ converges
to the twisted de Rham cohomology
\begin{equation}
\begin{array}{lclc}
\bigoplus\limits_{ \rm{p+q =1 }}E^{p, q}_{\infty}\cong H^o(M, H)
\quad\mbox{and}\quad \bigoplus\limits_{ \rm{p+q= 0}}E^{p,
q}_{\infty}\cong H^e(M, H).
\end{array}
\end{equation}
\end{proposition}

\begin{proof} Since the filtration is bounded and complete, the
proof follows from the standard algebraic topology method (see
\cite{M}). \end{proof}

\begin{remark}
\begin{enumerate}\item Note that
$$H_D^i(K_p)=\left\{\begin{array}{ll} H_D^e(K_p)&\mbox{$i$ is even,}\\
H_D^o(K_p)&\mbox{$i$ is odd.}
\end{array}
\right.$$ and $$
H_D^i(K_p/K_{p+1})=\left\{\begin{array}{ll} H_D^e(K_p/K_{p+1})&\mbox{$i$ is even,}\\
H_D^o(K_p/K_{p+1})&\mbox{$i$ is odd.}
\end{array}
\right.
$$
Then we have that $H_D^i(K_p)$ and $H_D^i(K_p/K_{p+1})$ are
$2$-periodic on $i$. Consequently, the spectral sequence
$\{E_r^{p,q}, d_r\}$ is $2$-periodic on $q$.

\item There is also a spectral sequence converging to the twisted
cohomology $H^*(M, E, H)$ for a flat vector bundle $E$ over $M$.
\end{enumerate}
\end{remark}

\begin{proposition}\label{e3}
For the spectral sequence in Proposition~\ref{specs},

${\rm (i)}$ The $E_1^{*, *}$-term is given by
\begin{equation*}
E_1^{p,q}=H_D^{p+q}(K_p/K_{p+1})\cong\left\{\begin{array}{ll}\Omega^p(M)
& ~~~\mbox{$q$
is even,} \\
0& ~~~\mbox{$q$ is odd.}\end{array}\right.
\end{equation*} and
$d_1x_p=dx_p$ for any $x_p\in E_1^{p,q}$.

${\rm (ii)}$ The $E_2^{*, *}$-term is given by
\begin{equation*}
E_2^{p,q}=H_{d_1}(E_1^{p, q})\cong\left\{\begin{array}{ll}H^p(M) &
~~~\mbox{$q$
is even,} \\
0& ~~~\mbox{$q$ is odd.}\end{array}\right.
\end{equation*}
and $d_2 = 0$.

${\rm (iii)}$ $E_3^{p, q}= E_2^{p, q}$ and $d_3[x_p]=[H_3\wedge
x_p]$ for $[x_p]_3\in E_3^{p, q}$.
\end{proposition}

\begin{proof} (i) By Lemma \ref{3.8}, we have the $E_1^{*,
*}$-term as desired, and by definition we obtain $d_1=j_1k_1:E_1^{p,q}\rightarrow
E_1^{p+1,q}$. We only need to consider the case when $q$ is even,
otherwise $d_1=0.$ By (\ref{3.2}) for odd $p$ (the case when $p$ is
even, is similar), we have a large commutative diagram
\begin{equation}\label{3.7}
\xymatrix{& &\vdots &\vdots &\vdots  & &\\
& 0\ar[r]& K_{p+1}^e\ar[u]_D\ar[r]^{i}& K_{p}^e\ar[u]_D\ar[r]^{j}&0 \ar[u]_D\ar[r]&0\\
&0\ar[r]&K_{p+1}^o\ar[u]_D\ar[r]^i& K_{p}^o\ar[u]_{D}\ar[r]^{j}
&\Omega^p(M) \ar[u]_D\ar[r]&0\\
&0\ar[r]& K_{p+1}^e\ar[u]_D\ar[r]^{i}& K_{p}^e\ar[u]_D\ar[r]^{j}&0 \ar[u]_D\ar[r]&0\\
& &\vdots \ar[u]_D &\vdots\ar[u]_D &\vdots \ar[u]_D & &\\}
\end{equation}
where the rows are exact and the columns are cochain complexes.

Let $x_p\in \Omega^p(M)\cong H_D^{p+q}(K_p/K_{p+1})\cong E_1^{p,q}$
and
\begin{equation}\label{4.9}
 x=\sum_{i=0}^{[\frac{n-p}{2}]}x_{p+2i}
\end{equation}
be an (inhomogeneous) form, where $x_{p+2i}$ is a $(p+2i)$-form ($0
\leq i\leq [\frac{n-p}{2}]$). Then $x\in K_p^o$, $jx=x_p$ and $Dx\in
K_p^e$. Also $Dx\in K_{p+1}^e$. By the definition of the
homomorphism $\delta$ in (\ref{3.3}), we have
\begin{equation}\label{3.10}
k_1x_p=[Dx]_D,
\end{equation}
 where $[\ \ ]_D$ is the
cohomology class in $H_D^{\ast}(K_{p+1})$. The class $[Dx]_D$ is
well defined and independent of the choices of $x_{p+2i}$ $(1 \leq
i\leq [\frac{n-p}{2}])$ (see \cite[p. 116]{Ha}).

Choose $x_{p+2i}=0$ ($1 \leq i\leq [\frac{n-p}{2}]$). Then we have
\begin{eqnarray*}
k_1x_p & = & [Dx]_D\\
&=&[dx_p+H\wedge x_p]_D\\
& = & [dx_p + \sum_{l=1}^{[\frac{n-1}{2}]}H_{2l+1}\wedge x_p]_D \in
H_D^{p+q+1}(K_{p+1}).
\end{eqnarray*}
Thus one obtains
\begin{eqnarray*}
d_1x_p &=&(j_1 k_1)x_p\\
&=& j_1(k_1(x_p))\\
& = & j_1[dx_p + \sum_{l=1}^{[\frac{n-1}{2}]}H_{2l+1}\wedge x_p]_D\\
& = & dx_p.
\end{eqnarray*}

(ii) By the definition of the spectral sequence and (i), one obtains
that $E_2^{p,q}\cong H^{p}(M)$ when $q$ is even, and $E_2^{p,q}=0$
when $q$ is odd. Note that $d_2: E^{p, q}_2 \to E_2^{p+2, q-1}$. It
follows that $d_2=0$ by degree reasons.

(iii) Note that $[x_p]_3\in E_3^{p, q}$ implies $dx_p=0$. Choose
$x_{p+2i}=0$ for $1 \leq i\leq [\frac{n-p}{2}]$, and we get
$$[Dx]_D=[H\wedge x_p]_D =[ \sum_{l=1}^{[\frac{n-1}{2}]}H_{2l+1}\wedge x_p]_D\in
H_D^{p+q+1}(K_{p+1}),$$ where $x$ is given in the proof of (i). Note
that
\begin{equation}\label{3.11}
\xymatrix{
&H_D^{p+q+1}(K_{p+1})&\ar[l]_{i_1^2} H_D^{p+q+1}(K_{p+3})\ar[r]^{j_1}&H_D^{p+q+1}(K_{p+3}/K_{p+4})\\
&[Dx]_D\ar[r]^{(i_1^{-1})^2}&[Dx]_D\ar[r]^{j_1}
& H_3\wedge x_p. \\
}
\end{equation}
It follows that
\begin{equation}\label{3.12}
\begin{array}{lll}
d_3[x_p]_3 & =& j_3k_3[x_p]_3\\

&=&j_3(k_1x_p)\\

& =& j_3[Dx]_D \\

& =& [j_1((i_1^{-1})^2[Dx]_D)]_3\\

&=& [H_3\wedge x_p]_3,
\end{array}
\end{equation}
 where the first,
second and fourth identities follow from the definitions of $d_3$,
$k_3$ and $j_3$ respectively, and the third and the last identities
follow from (\ref{3.10}) and (\ref{3.11}). By (ii), $d_2=0$, so
$E_3^{p,q}=E_2^{p,q}.$ Then we have
$$d_3[x_p]=[H_3\wedge x_p].$$
\end{proof}

\begin{corollary}\label{4.6}
$d_{2k}=0$ for $k\geq 1.$ Therefore, for $k\geq 1$,
 \begin{equation} \label{4.7}
E_{2k+1}^{p,q}=E_{2k}^{p,q}.
\end{equation}
\end{corollary}

\begin{proof} Note that $d_{2k}:E_{2k}^{p,q}\longrightarrow
E_{2k}^{p+2k,q+1-2k}.$ By Proposition~\ref{e3} (ii), if $q$ is odd,
then $E_{2}^{p,q}=0$ which implies that $E_{2k}^{p,q}=0.$ By degree
reasons, we have $d_{2k}=0$ and $E_{2k+1}^{p,q}=E_{2k}^{p,q}$ for
$k\geq 1.$ \end{proof}

The differential $d_3$ for the case in which $H=H_3$ is shown in
\cite[\S 6]{A-S}, and the $E_2^{p,q}$-term is also known.

\section{Differentials $d_{2t+3}$ $(t\geq 1)$ in terms of cup products}

In this section, we will show that the differentials $d_{2t+3}$
$(t\geq 1)$ can be given in terms of cup products.

We first consider the general case that
$H=\sum_{i=1}^{[\frac{n-1}{2}]}H_{2i+1}$. For $[x_p]_{2t+3}\in
E_{2t+3}^{p,q}$, we let $x=\sum_{j=0}^{[\frac{n-p}{2}]}x_{p+2j}\in
F_p({\Omega}^*(M))$. Then we have
\begin{equation}\label{6.6}
\begin{array}{ll}
Dx&=(d+\sum\limits_{i= 1}^{[\frac{n-1}{2}]}H_{2i+1})(\sum\limits_{j=0}^{[\frac{n-p}{2}]}x_{p+2j})\\

&=dx_p+\sum\limits_{j=
0}^{[\frac{n-p}{2}]-1}(dx_{p+2j+2}+\sum\limits_{i=1}^{j+1}H_{2i+1}\wedge
x_{p+2(j-i)+2}).
\end{array}
\end{equation}
Denote $y=Dx =\sum_{j=0}^{[\frac{n-p}{2}]}y_{p+2j+1}$, where
\begin{equation}\label{6.7}
\left\{
\begin{array}{ll}
&y_{p+1}=dx_p, \\

&y_{p+2j+3}=dx_{p+2j+2}+\sum\limits_{i=1}^{j+1} H_{2i+1}\wedge
x_{p+2(j-i)+2}\quad (0\leq j\leq
[\frac{n-p}{2}]-1).\\
\end{array}\right.
\end{equation}

\begin{theorem}\label{4.1}
For $[x_p]_{2t+3}\in E_{2t+3}^{p,q}$ $(t\geq1)$, there exist
$x_{p+2i}=x_{p+2i}^{(t)}$ $(1\leq i\leq t)$ such that $y_{p+2j+1}=0$
$(0\leq j\leq t)$ and
\begin{equation*}
d_{2t+3}[x_p]_{2t+3}=[\sum\limits_{i=1}^t
H_{2i+1}\wedge x_{p+2(t-i)+2}^{(t)}+H_{2t+3}\wedge x_p]_{2t+3},\\
\end{equation*}
where the $(p+2i)$-form $x_{p+2i}^{(t)}$ depends on $t$.

\end{theorem}

\begin{proof} The theorem is shown by mathematical
induction on $t$.

When $t=1$, $[x_p]_{2t+3}=[x_p]_5$. $[x_p]_5\in E_5^{p,q}$ implies
that $dx_p=0$ and $d_3[x_p]=[H_3\wedge x_p]=0$ by
Proposition~\ref{e3}. Thus there exists a $(p+2)$-form $v_1$ such
that $H_3\wedge x_p=d(-v_1).$ We can choose $x_{p+2}^{(1)}=v_1$ to
get $y_{p+3}=dx_{p+2}^{(1)}+H_3\wedge x_p=dv_1+H_3\wedge x_p=0$ from
(\ref{6.7}). Note that
\begin{equation}\label{6.2}
\xymatrix{
&H_D^{p+q+1}(K_{p+1})&\ar[l]_{i_1^4} H_D^{p+q+1}(K_{p+5})\ar[r]^{j_1}&H_D^{p+q+1}(K_{p+5}/K_{p+6})\\
&[Dx]_D\ar[r]^{(i_1^{-1})^4}&[Dx]_D\ar[r]^{j_1}
& y_{p+5}, \\
}
\end{equation}
we obtain
\begin{equation}\label{6.3}
\begin{array}{ll}
d_5[x_p]_5&=j_5k_5[x_p]_5\\
&=j_5(k_1x_p)\\
&=j_5[Dx]_D\\

&=[j_1(i_1^{-1})^4[Dx]_D]_5\\

&=[y_{p+5}]_5.
\end{array}
\end{equation}
The reasons for the identities in (\ref{6.3}) are similar to those
of (\ref{3.12}). Thus we have
\begin{equation*}
\begin{array}{ll}
d_5[x_p]_5&=[dx_{p+4}+H_3\wedge x_{p+2}^{(1)}+H_5\wedge x_p]_5\\

&=[H_3\wedge x_{p+2}^{(1)}+H_5\wedge x_p]_5,
\end{array}
\end{equation*}
where the first identity follows from (\ref{6.3}) and the definition
of $y_{p+5}$ in (\ref{6.7}), and the second one follows from that
$dx_{p+4}$ vanishes in $E_5^{\ast,\ast}$. Hence the result holds for
$t=1$.

Suppose the result holds for $t\leq m-1$. Now we show that the
theorem also holds for $t=m$.

From $[x_p]_{2m+3}\in E_{2m+3}^{p,q}$, we have $[x_p]_{2m+1}\in
E_{2m+1}^{p,q}$ and $d_{2m+1}[x_p]_{2m+1}=0$. By induction, there
exist $x_{p+2i}^{(m-1)}$ ($1\leq i \leq m-1$) such that

\begin{equation}\label{4.8}
\left\{ \begin{array}{ll}&y^{(m-1)}_{p+1}(x_p)=dx_p=0,\\

&y^{(m-1)}_{p+3}(x_p) = dx_{p+2}^{(m-1)}+H_3\wedge x_p = 0,\\

&y^{(m-1)}_{p+2i+1}(x_p) =
dx_{p+2i}^{(m-1)}+\sum_{j=1}^{i-1}H_{2j+1}\wedge
x_{p+2(i-j)}^{(m-1)}+H_{2i+1}\wedge x_p\\

&\quad \quad \quad\quad\quad\ = 0~~ (2\leq i\leq m-1),\\

&d_{2m+1}[x_p]_{2m+1}=[\sum\limits_{i=1}^{m-1} H_{2i+1}\wedge
x_{p+2(m-i)}^{(m-1)}+H_{2m+1}\wedge x_p]_{2m+1}=0.
\end{array}\right.
\end{equation}

By $d_{2m}=0$ and the last equation in (\ref{4.8}), there exists a
$(p+2)$-form $w_{p+2}$ such that
\begin{equation}\label{4.3}
[\sum\limits_{i=1}^{m-1} H_{2i+1}\wedge
x_{p+2(m-i)}^{(m-1)}+H_{2m+1}\wedge
x_p]_{2m-1}=d_{2m-1}[w_{p+2}]_{2m-1}.
\end{equation}

By induction and $[w_{p+2}]_{2m-1}\in E_{2m-1}^{p+2,q-2}$, there
exist $w^{(m-2)}_{p+2(i+1)}$ $(1\leq i \leq m-2)$ such that

\begin{equation}\label{4.4}
\left\{ \begin{array}{ll}&y^{(m-2)}_{p+3}(w_{p+2})=dw_{p+2}=0,\\

&y^{(m-2)}_{p+5}(w_{p+2})=dw^{(m-2)}_{p+4}+H_3\wedge w_{p+2}=0,\\

&y^{(m-2)}_{p+2i+3}(w_{p+2}) =
dw^{(m-2)}_{p+2(i+1)}+\sum_{j=1}^{i-1}H_{2j+1}\wedge
w^{(m-2)}_{p+2(i-j+1)}+ H_{2i+1}\wedge w_{p+2}\\

 &\quad\quad\quad\quad\quad\quad\ \ = 0~~ (2\leq i\leq m-2),\\

&d_{2m-1}[w_{p+2}]_{2m-1}=[\sum\limits_{i=1}^{m-2} H_{2i+1}\wedge
w^{(m-2)}_{p+2(m-i)}+H_{2m-1}\wedge w_{p+2}]_{2m-1}.
\end{array}\right.
\end{equation}
By (\ref{4.3}) and the last equation in (\ref{4.4}), we obtain
$$[\sum\limits_{i=1}^{m-2} H_{2i+1}\wedge
(x_{p+2(m-i)}^{(m-1)}- w^{(m-2)}_{p+2(m-i)})+H_{2m-1}\wedge
(x_{p+2}^{(m-1)}-w_{p+2})+H_{2m+1}\wedge x_p]_{2m-1}=0.$$ Note that
$d_{2m-2}=0$, it follows that there exists a $(p+4)$-form $w_{p+4}$
such that
\begin{displaymath}
\begin{array}{ll}
 &[\sum\limits_{i=1}^{m-2} H_{2i+1}\wedge (x_{p+2(m-i)}^{(m-1)}-
w^{(m-2)}_{p+2(m-i)})+H_{2m-1}\wedge
(x_{p+2}^{(m-1)}-w_{p+2})+H_{2m+1}\wedge x_p]_{2m-3}\\

&=d_{2m-3}[w_{p+4}]_{2m-3}.
\end{array}
\end{displaymath}

Keeping the same iteration process as mentioned above, we have
\begin{equation*}
\begin{array}{ll}
&[\sum\limits_{i=1}^{2}(H_{2i+1}\wedge
(x_{p+2(m-i)}^{(m-1)}-\sum\limits_{j=1}^{m-3}w_{p+2(m-i)}^{(m-1-j)}))+\\

&\sum\limits_{i=3}^{m-1}(H_{2i+1}\wedge
(x_{p+2(m-i)}^{(m-1)}-\sum\limits_{j=1}^{
m-1-j}w_{p+2(m-i)}^{(m-1-j)}-w_{p+2(m-i)}))+H_{2m+1}\wedge x_p]_7=0.
\end{array}
\end{equation*}
By $d_6=0$, it follows that there exists a $(p+2(m-2))$-form
$w_{p+2(m-2)}$ such that
\begin{equation}\label{liu}
\begin{array}{ll}
&[\sum\limits_{i=1}^{2}(H_{2i+1}\wedge
(x_{p+2(m-i)}^{(m-1)}-\sum\limits_{j=1}^{m-3}w_{p+2(m-i)}^{(m-1-j)}))+
\sum\limits_{i=3}^{m-1}(H_{2i+1}\wedge
(x_{p+2(m-i)}^{(m-1)}-\\

&\sum\limits_{j=1}^{
m-i-1}w_{p+2(m-i)}^{(m-1-j)}-w_{p+2(m-i)}))+H_{2m+1}\wedge
x_p]_5=d_5[w_{p+2(m-2)}]_5.
\end{array}
\end{equation}

By induction and $[w_{p+2(m-2)}]_5\in E_5^{p+2(m-2),q-2(m-2)}$,
there exists  $w_{p+2(m-1)}^{(1)}$ such that
\begin{equation}\label{liu2}
\left\{\begin{array}{ll} &y_{p+2m-3}^{(1)}(w_{p+2(m-2)})=dw_{p+2(m-2)}=0,\\

&y_{p+2m-1}^{(1)}(w_{p+2(m-2)})=dw_{p+2(m-1)}^{(1)}+H_3\wedge w_{p+2(m-2)}=0,\\

&d_5[w_{p+2(m-2)}]_5=[H_3\wedge w_{p+2(m-1)}^{(1)}+H_5\wedge
w_{p+2(m-2)}]_5.
\end{array}\right.
\end{equation}
By (\ref{liu}), the last equation in (\ref{liu2}) and $d_4=0$, it
follows that there exists a $(p+2(m-1))$-form $w_{p+2(m-1)}$ such
that
\begin{equation*}
\begin{array}{ll}
&[(H_3\wedge
(x_{p+2(m-1)}^{(m-1)}-\sum\limits_{j=1}^{m-2}w_{p+2(m-1)}^{(m-1-j)}))+
\sum\limits_{i=2}^{m-1}(H_{2i+1}\wedge (x_{p+ 2(m-i)}^{(m-1)}-\\

&\sum\limits_{j=1}^{m-i-1}w_{p+2(m-i)}^{(m-1-j)}-w_{p+2(m-i)}))+H_{2m+1}\wedge
x_p]=d_3[w_{p+2(m-1)}]=[H_3\wedge w_{p+2(m-1)}]

\end{array}
\end{equation*}
and $y_{p+2m-1}^{(0)}(w_{p+2(m-1)})=dw_{p+2(m-1)}=0$. Thus there
exists a $(p+2m)$-form $w_{p+2m}$ such that
\begin{equation}\label{4.5}
\sum\limits_{i=1}^{m-1}(H_{2i+1}\wedge
(x_{p+2(m-i)}^{(m-1)}-\sum\limits_{j=1}^{m-i-1}w_{p+2(m-i)}^{(m-1-j)}-w_{p+2(m-i)}))+H_{2m+1}\wedge
x_p=dw_{p+2m}.
\end{equation}
By comparing (\ref{4.5}) with (\ref{6.7}), we choose at this time
\begin{equation}\label{liu3}
\left\{\begin{array}{ll} &x_{p+2}=x_{p+2}^{(m)}=x_{p+2}^{(m-1)}-w_{p+2},\\

&x_{p+2i}=x_{p+2i}^{(m)}=x_{p+2i}^{(m-1)}-\sum\limits_{j=1}^{i-1}w_{p+2i}^{(m-1-j)}-w_{p+2i}\quad
(2\leq i
\leq m-1),\\

&x_{p+2m}=x_{p+2m}^{(m)}=-w_{p+2m}.
\end{array}\right.
\end{equation}
From (\ref{6.7}), by a direct computation we have
\begin{equation}\label{liu4}
\left\{\begin{array}{ll} &y_{p+1}=y_{p+1}^{(m-1)}(x_p)=0,\\

&y_{p+2i-1}=y_{p+2i-1}^{(m-1)}(x_p)-\sum\limits_{j=1}^{i-1}y^{(m-1-j)}_{p+2i-1}(w_{p+2j})=0
\quad
(2\leq i \leq m),\\

&y_{p+2m+1}=0.
\end{array}\right.
\end{equation}
Note that
\begin{equation}\label{6.4}
\xymatrix{
&H_D^{p+q+1}(K_{p+1})&\ar[l]_{i_1^{2(m+1)}} H_D^{p+q+1}(K_{p+2m+3})\ar[r]^<<<<{j_1}&H_D^{p+q+1}(K_{p+2m+3}/K_{p+2m+4})\\
&[Dx]_D\ar[r]^{(i_1^{-1})^{2(m+1)}}&[Dx]_D\ar[r]^{j_1}
& y_{p+2m+3}. \\
}
\end{equation}
By the similar reasons as in (\ref{3.12}), the following identities
hold.
\begin{equation}\label{6.5}
\begin{array}{ll}
d_{2m+3}[x_p]_{2m+3}&=j_{2m+3}k_{2m+3}[x_p]_{2m+3}\\

&=j_{2m+3}(k_1x_p)\\

&=j_{2m+3}[Dx]_D\\

&=[j_1(i_1^{-1})^{2(m+1)}[Dx]_D]_{2m+3}\\

&=[y_{p+2m+3}]_{2m+3}.
\end{array}
\end{equation}
So we have
\begin{equation*}
\begin{array}{ll}
d_{2m+3}[x_p]_{2m+3}&=[y_{p+2m+3}]_{2m+3}\\

&=[dx_{p+2m+2}+\sum\limits_{i=1}^{m} H_{2i+1}\wedge
x_{p+2(m-i+1)}^{(m)}+H_{2m+3}\wedge x_p]_{2m+3} ~{\rm by ~(\ref{6.7}) }\\

&=[\sum\limits_{i=1}^{m} H_{2i+1}\wedge
x_{p+2(m-i+1)}^{(m)}+H_{2m+3}\wedge x_p]_{2m+3},
\end{array}
\end{equation*}
showing that the result also holds for $t=m$.

The proof of the theorem is finished.
\end{proof}

\begin{remark}
Note that $x_{p+2i}^{(t)}$ ($1\leq i\leq t$) depends on $t$, and
$x_{p+2i}^{(t_1)}\neq x_{p+2i}^{(t_2)}$ on the condition that
$t_1\neq t_2$ generally. $x_{p+2i}^{(t)}$ ($1\leq i\leq t$) are
related to $x_{p+2j}^{(t-1)}$ ($1\leq j\leq t-1$, $j\leq i$).

\end{remark}

Now we consider the special case in which $H=H_{2s+1}$ $(s\geq 1)$
only. For this special case, we will give a more explicit result
which is stronger than Theorem \ref{4.1}.

For $x=\sum_{j= 0}^{[\frac{n-p}{2}]}x_{p+2j}$, we have
\begin{equation*}
\begin{array}{ll}
Dx&=(d+H_{2s+1})(\sum_{j= 0}^{[\frac{n-p}{2}]}x_{p+2j})\\

&=\sum_{j=0}^{s-1}dx_{p+2j}+\sum_{j=s}^{[\frac{n-p}{2}]}(dx_{p+2j}+H_{2s+1}\wedge
x_{p+2(j-s)}).
\end{array}
\end{equation*}
Denote
\begin{equation}\label{6.6}
\left\{
\begin{array}{ll}
y_{p+2j+1}=dx_{p+2j}&(0\leq j\leq s-1),\\

y_{p+2j+3}=dx_{p+2j+2}+H_{2s+1}\wedge x_{p+2(j-s)+2}&(s-1\leq j\leq [\frac{n-p}{2}]-1).\\
\end{array}\right.
\end{equation}
Then $Dx=\sum_{j= 0}^{[\frac{n-p}{2}]}y_{p+2j+1}$.

\begin{theorem}\label{4.2}
For $H=H_{2s+1}$ $(s\geq 1)$ only and $[x_p]_{2t+3}\in
E_{2t+3}^{p,q}$ $(t\geq 1)$, there exist
$x_{p+2is}=x_{p+2is}^{([\frac{t}{s}])}$, $x_{p+2(i-1)s+2j}=0$ and
$x_{p+2[\frac{t}{s}]s+2k}=0$ for $1\leq i\leq [\frac{t}{s}],$ $1\leq
j\leq s-1$ and $1\leq k \leq t-[\frac{t}{s}]s$ such that
$y_{p+2u+1}=0$ $(0\leq u\leq t)$ and
\begin{equation*}
 d_{2t+3}[x_p]_{2t+3}=\left\{\begin{array}{lll}&[H_{2s+1}\wedge x_p]_{2s+1} \quad&t=s-1,\\

 &[H_{2s+1}\wedge x_{p+2(l-1)s}^{(l-1)}]_{2t+3} \quad &t=ls-1~~ (l\geq 2), \\
&0 \quad & otherwise,\end{array}\right.
\end{equation*}
where the $(p+2is)$-form $x_{p+2is}^{([\frac{t}{s}])}$ depends on
$[\frac{t}{s}]$.

\end{theorem}

\begin{proof} The proof of the theorem is by mathematical induction
on $s$.

When $s=1$, the result follows from Theorem \ref{4.1}.

When $s\geq 2$, we prove the result by mathematical induction on
$t$. We first show that the result holds for $t=1$. Note that
$[x_p]_5\in E_5^{p,q}$ implies $y_{p+1}=dx_p=0.$ Choose $x_{p+2}=0$
and make $y_{p+3}=0$.

{\bf (i).}~~When $s=2$, by (\ref{6.3}) we have
\begin{equation*}
\begin{array}{ll}
d_5[x_p]_5&=[y_{p+5}]_5\\

&=[dx_{p+4}+H_5\wedge x_p]_5\\

&=[H_5\wedge x_p]_5.
\end{array}
\end{equation*}

{\bf (ii).}~~When $s\geq 3$, by (\ref{6.3}) we have
$$d_5[x_p]_5=[y_{p+5}]_5=[dx_{p+4}]_5=0.$$

Combining (i) and (ii), we have that the theorem holds for $t=1$.

Suppose the theorem holds for $t\leq m-1$. Now we show that the
theorem also holds for $t=m$.

{\bf Case 1.}~~$2\leq m\leq s-1$.

By induction, the theorem holds for $1\leq t\leq m-1$. Choose
$x_{p+2i}=0$ $(1\leq i\leq m)$, and from (\ref{6.6}) we easily get
that $y_{p+2j+1}=0$ $(0\leq j\leq m)$. By (\ref{6.5}) and
(\ref{6.6}), we have
\begin{equation*}
\begin{array}{ll}
d_{2m+3}[x_p]_{2m+3}&=[y_{p+2m+3}]_{2m+3}\\

&=\left\{\begin{array}{lll}[dx_{p+2(m+1)}]_{2m+3}&2\leq m\leq s-2,\\

[dx_{p+2(m+1)}+H_{2s+1}\wedge x_{p}]_{2m+3}& m=s-1,\end{array}\right.\\

&=\left\{\begin{array}{ll}0 &2\leq m\leq s-2,\\

[H_{2s+1}\wedge x_p]_{2s+1}&m=s-1.
\end{array}\right.
\end{array}
\end{equation*}

{\bf Case 2.}~~$m=ls-1$~~ $(l\geq 2)$.

By induction, the theorem holds for $t=m-1=ls-2$. Thus, there exist
$x_{p+2is}=x_{p+2is}^{([\frac{m-1}{s}])}=x_{p+2is}^{(l-1)}$,
$x_{p+2(i-1)s+2j}=0$ and $x_{p+2(l-1)s+2k}=0$ for $1\leq i\leq l-1,$
$1\leq j\leq s-1$ and $1\leq k \leq s-2$ such that $y_{p+2u+1}=0$
$(0\leq u\leq ls-2)$. Choose $x_{p+2(ls-1)}=0$, and by (\ref{6.6})
we get
$$\begin{array}{ll}
y_{p+2(ls-1)+1}&=dx_{p+2(ls-1)}+H_{2s+1}\wedge
x_{p+2(l-1)s-2}\\&=0+H_{2s+1}\wedge 0\\&=0.\end{array}$$ Then we
have
\begin{equation*}
\begin{array}{ll}
d_{2(ls-1)+3}[x_p]_{2(ls-1)+3}&=[y_{p+2ls+1}]_{2(ls-1)+3}\quad {\rm by~ (\ref{6.5})}\\

&=[dx_{p+2ls}+H_{2s+1} \wedge
x_{p+2(l-1)s}^{(l-1)}]_{2(ls-1)+3}\quad {\rm by~ (\ref{6.6})}\\

&=[H_{2s+1}\wedge x_{p+2(l-1)s}^{(l-1)}]_{2(ls-1)+3}.
\end{array}
\end{equation*}

{\bf Case 3.}~~$m=ls$ ~~($l\geq 1$).

By induction, there exist
$x_{p+2is}=x_{p+2is}^{([\frac{ls-1}{s}])}=x_{p+2is}^{(l-1)}$,
$x_{p+2(i-1)s+2j}=0$ and $x_{p+2(l-1)s+2k}=0$ for $1\leq i\leq l-1$,
$1\leq j\leq s-1$ and $1\leq k \leq s-1$ such that
 $y_{p+2u+1}=0$ $(0\leq
u\leq ls-1)$. By the same method as in Theorem \ref{4.1}, one has
that there exist $x_{p+2is}=x_{p+2is}^{(l)}$, $x_{p+2(i-1)s+2j}=0$
and $x_{p+2(l-1)s+2k}=0$ for $1\leq i\leq l$, $1\leq j\leq s-1$ and
$1\leq k \leq s-1$ such that $y_{p+2u+1}=0$ $(0\leq u\leq ls)$. By
(\ref{6.5}), (\ref{6.6}) and $x_{p+2ls-2s+2}=0$, we have
$$\begin{array}{ll}
d_{2ls+3}[x_p]_{2ls+3}&=[y_{p+2ls+3}]_{2ls+3}\\
&=[dx_{p+2ls+2}+H_{2s+1} \wedge x_{p+2ls-2s+2}]_{2ls+3}\\
&=0. \end{array}$$

{\bf Case 4.}~~$ls<m<(l+1)s-1$ ($l\geq 1$).

By induction, there exist
$x_{p+2is}=x_{p+2is}^{([\frac{m-1}{s}])}=x_{p+2is}^{(l)}$,
$x_{p+2(i-1)s+2j}=0$ and $x_{p+2ls+2k}=0$ for $1\leq i\leq l,$
$1\leq j\leq s-1$ and $1\leq k \leq m-ls-1$ such that
 $y_{p+2u+1}=0$ $(0\leq u\leq m-1)$. Choose $x_{p+2m}=0$
and make $y_{p+2m+1}=0$. By (\ref{6.5}), (\ref{6.6}) and
$x_{p+2m-2s+2}=0$, we have
$$\begin{array}{ll}d_{2m+3}[x_p]_{2m+3}&=[y_{p+2m+3}]_{2m+3}\\
&=[dx_{p+2m+2}+H_{2s+1} \wedge x_{p+2m-2s+2}]_{2m+3}\\&=0.
\end{array}$$

Combining Cases 1-4, we have that the result holds for $t=m$ and the
proof is completed. \end{proof}

\begin{remark}
\begin{enumerate}

\item

Theorems \ref{4.1} and \ref{4.2} show that the differentials in the
spectral sequence (\ref{ss}) can be computed in terms of cup
products with $H_{2i+1}$'s. The existence of $x_{p+2i}^{(t)}$'s and
$x_{p+2is}^{([\frac{t}{s}])}$'s in Theorems \ref{4.1} and \ref{4.2}
plays an essential role in proving Theorems \ref{differentials1} and
\ref{differentials2}, respectively. Theorems \ref{4.1} and \ref{4.2}
give a description of the differentials at the level of
$E_{2t+3}^{p,q}$ for the spectral sequence (\ref{ss}), which was
ignored in the previous studies of the twisted de Rham cohomology in
\cite{A-S, M-W}.

\item Note that Theorem \ref{4.2} is not a corollary of Theorem
\ref{4.1}, and it can not be obtained from Theorem \ref{4.1}
directly.

\end{enumerate}
\end{remark}

\section{Differentials $d_{2t+3}$ ($t\geq 1$) in terms of Massey products}

The Massey product is a cohomology operation of higher order
introduced in \cite{Ma2}, which generalizes the cup product. In
\cite{May}, May showed that the differentials in the Eilenberg-Moore
spectral sequence associated with the path-loop fibration of a path
connected, simply connected space are completely determined by
higher order Massey products. Kraines and Schochet \cite{K-S} also
described the differentials in Eilenberg-Moore spectral sequence by
Massey products. In order to describe the differentials $d_{2t+3}$
($t\geq 1$) in terms of Massey products, we first recall briefly the
definition of Massey products (see \cite{K, May, May2, M}). Then the
main theorems in this paper will be shown.

Because of different conventions in the literature used to define
Massey products, we present the following definitions. If $x\in
\Omega^{p}(M)$, the symbol $\bar{x}$ will denote $(-1)^{1+\hbox{deg}
x} x=(-1)^{1+p}x$. We first define the Massey triple product.

Let $x_1$, $x_2$, $x_3$ be closed differential forms on $M$ of
degrees $r_1$, $r_2$, $r_3$ with $[x_1][x_2]=0$ and $[x_2][x_3]=0,$
where [ \ ] denotes the de Rham cohomology class. Thus, there are
differential forms $v_1$ of degree $r_1+r_2-1$ and $v_2$ of degree
$r_2+r_3-1$ such that $dv_1=\bar{x}_1\wedge x_2$ and
$dv_2=\bar{x}_2\wedge x_3$. Define the $(r_1+r_2+r_3-1)$-form
\begin{equation}\label{5.00}
\omega=\bar{v}_1\wedge x_3+\bar{x}_1\wedge v_2.
\end{equation}
Then $\omega$ satisfies
\begin{equation*}
\begin{array}{ll}
 d(\omega)&=(-1)^{r_1+r_2}dv_1\wedge
x_3+(-1)^{r_1}\bar{x}_1\wedge dv_2\\
&=(-1)^{r_1+r_2}\bar{x}_1\wedge x_2\wedge
x_3+(-1)^{r_1+r_2+1}\bar{x}_1\wedge x_2\wedge x_3\\&=0.
\end{array}
\end{equation*}
Hence a set of all the cohomology classes $[\omega]$ obtained by the
above procedure is defined to be the {Massey triple product $\langle
x_1,x_2,x_3\rangle$} of $x_1, x_2$ and $x_3$. Due to the ambiguity
of $v_i, i=1,2$,  the Massey triple product $\langle
x_1,x_2,x_3\rangle$ is a representative of the quotient group
$$H^{r_1+r_2+r_3-1}(M)/([x_1]H^{r_2+r_3-1}(M)+H^{r_1+r_2-1}(M)[x_3]).$$

\begin{definition}\label{massey} Let $(\Omega^{\ast}(M),d)$ be de Rham complex, and $x_1$, $x_2$,
$\cdots$, $x_n$  closed differential forms on $M$ with $[x_i]\in
H^{r_i}(M)$. A collection of forms, $A = (a_{i,j})$, for $1\leq
i\leq j \leq k$ and $(i,j)\neq(1,n)$ is said to be a {defining
system for the $n$-fold Massey product $\langle
x_1,x_2,\cdots,x_n\rangle$} if
\begin{enumerate}
\item $a_{i,j}\in \Omega^{r_i+r_{i+1}+\cdots+r_j-j+i}(M),$

\item $a_{i,i}=x_i$ for $i=1,2,\cdots, k$,

\item $d(a_{i,j})=\sum
\limits_{r=i}^{j-1}\bar{a}_{i,r}\wedge a_{r+1,j}.$\end{enumerate}
The $(r_1+\cdots+r_n-n+2)$-dimensional cocycle, $c(A)$, defined by
\begin{equation}\label{5.02}c(A) = \sum\limits_{r=1}^{n-1}\bar{a}_{1,r}\wedge
a_{r+1,n}\in \Omega^{r_1+\cdots+r_n-n+2}(M)
\end{equation} is called the {
related cocycle of the defining system $A$}.
\end{definition}

\begin{remark}
There is a unique matrix associated to each defining system $A$ as
follows.
\begin{displaymath}
\left(
  \begin{array}{ccccccc}
    a_{1,1} & a_{1,2} & a_{1,3} & \cdots & a_{1,n-2} & a_{1,n-1} &  \\
      & a_{2,2} & a_{2,3} & \cdots & a_{2,n-2} & a_{2,n-1} & a_{2,n} \\
      &   & a_{3,3} & \cdots & a_{3,n-2} & a_{3,n-1} & a_{3,n} \\
      &   &   & \ddots & \vdots & \vdots & \vdots \\
      &   &   &   & a_{n-2,n-2} & a_{n-2,n-1} & a_{n-2,n} \\
      &   &   &   &   & a_{n-1,n-1} & a_{n-1,n} \\
      &   &   &   &   &   & a_{n,n} \\
  \end{array}
\right)_{n\times n.} \end{displaymath}
\end{remark}

\begin{definition}\label{Massey2}
The $n$-fold Massey product $\langle x_1,x_2,\cdots,x_n\rangle$ is
said to be defined if there is a defining system for it. If it is
defined, then $\langle x_1,x_2,\cdots,x_n\rangle$ consists of all
classes $w\in H^{r_1+r_2+\cdots +r_n-n+2}(M)$ for which there exists
a defining system $A$ such that $c(A)$ represents $w$.
\end{definition}

\begin{remark}
There is an inherent ambiguity in the definition of the Massey
product arising from the choices of defining systems. In general,
the $n$-fold Massey product may or may not be a coset of a subgroup,
but its indeterminacy is a subset of a matrix Massey product (see
\cite[\S 2]{May}).
\end{remark}

Based on Theorems \ref{4.1} and \ref{4.2}, we have the following
lemma on defining systems for the two Massey products we consider in
this paper.

\begin{lemma}\label{ds}

{\rm (1)} For $[x_p]_{2t+3}\in E_{2t+3}^{p,q}$ ($t\geq 1$), there
are defining systems for $\langle
\underbrace{H_3,\cdots,H_3}\limits_{t+1},x_p\rangle$ obtained from
Theorem \ref{4.1}.

{\rm (2)} For $[x_p]_{2t+3}\in E_{2t+3}^{p,q}$, when $t=ls-1$
($l\geq 2$) there are defining systems for $\langle
\underbrace{H_{2s+1},\cdots,H_{2s+1}}\limits_{l},x_p\rangle$
obtained from Theorem \ref{4.2} .
\end{lemma}

\begin{proof}
(1) From Theorem \ref{4.1}, there exist $x_{p+2j}^{(t)}$ $(1\leq
j\leq t)$ such that $y_{p+2i+1}=0$ $(0\leq i \leq t)$ and
$d_{2t+3}[x_p]_{2t+3}=[\sum\limits_{i=1}^t H_{2i+1}\wedge
x_{p+2(t-i+1)}^{(t)}+H_{2t+3}\wedge x_p]_{2t+3}.$ By Theorem
\ref{4.1} and (\ref{6.7}), there exists a defining system
$A=(a_{i,j})$ for $\langle
\underbrace{H_3,\cdots,H_3}\limits_{t+1},x_p\rangle$ as follows:
\begin{equation}\label{d1}
\left\{\begin{array}{lll}&a_{t+2,t+2}=x_p,\\

&a_{i,i+k}=(-1)^k H_{2k+3}~~&(1\leq i\leq t+1-k,~ 0\leq k < t),\\

&a_{i,t+2}=(-1)^{t+2-i}x_{p+2(t+2-i)}^{(t)}~~&(2\leq i\leq t+1),

\end{array}\right.
\end{equation} to which the matrix associated
is given by
\begin{equation}\label{big} \left(
  \begin{array}{ccccccc}
    H_3 & -H_5 & H_7 & \cdots & (-1)^{t-1}H_{2t+1} & (-1)^tH_{2t+3} &  \\
      & H_3 & -H_5 & \cdots & (-1)^{t-2}H_{2t-1} & (-1)^{t-1}H_{2t+1} & (-1)^tx_{p+2t}^{(t)} \\
      &   & H_3 & \cdots & (-1)^{t-3}H_{2t-3} & (-1)^{t-2}H_{2t-1} & (-1)^{t-1}x_{p+2t-2}^{(t)} \\
      &   &   & \ddots & \vdots & \vdots & \vdots \\
      &   &   &   & H_3 & -H_5 & (-1)^2x_{p+4}^{(t)} \\
      &   &   &   &   & H_3 & -x_{p+2}^{(t)} \\
      &   &   &   &   &   & x_p \\
  \end{array}
\right)_{(t+2)\times (t+2).} \end{equation} The desired result
follows.

(2) By Theorem \ref{4.2}, there exist $x_{p+2is}=x_{p+2is}^{(l-1)}$,
$x_{p+2(i-1)s+2j}=0$ and $x_{p+2(l-1)s+2k}=0$ for $1\leq i\leq l-1,$
$1\leq j\leq s-1$ and $1\leq k \leq s-1$ such that $y_{p+2i+1}=0$
($0\leq i \leq t$) and $d_{2t+3}[x_p]_{2t+3}=[H_{2s+1}\wedge
x_{p+2(l-1)s}^{(l-1)}]_{2t+3}.$ By Theorem \ref{4.2} and
(\ref{6.6}), there also exists a defining system $A=(a_{i,j})$ for
$\langle
\underbrace{H_{2s+1},\cdots,H_{2s+1}}\limits_{l},x_p\rangle$ as
follows:
\begin{equation}\label{d2}
\left\{\begin{array}{lll}&a_{i,j}=0~~&(1\leq i<j\leq l),\\

&a_{i,i}=H_{2s+1}~~&(1\leq i\leq l),\\

&a_{l+1,l+1}=x_p,\\

&a_{i,l+1}=(-1)^{l+1-i}x_{p+2(l+1-i)s}^{(l-1)}~~&(2\leq i\leq l),

\end{array}\right.
\end{equation}
to which the matrix associated is given by
\begin{equation}\label{small}
\left(
  \begin{array}{ccccccc}
    H_{2s+1} & 0 & 0 & \cdots & 0 & 0 &  \\
      & H_{2s+1} & 0 & \cdots & 0 & 0 & (-1)^{l-1}x_{p+2(l-1)s}^{(l-1)} \\
      &   & H_{2s+1} & \cdots & 0 & 0 & (-1)^{l-2}x_{p+2(l-2)s}^{(l-1)} \\
      &   &   & \ddots & \vdots & \vdots & \vdots \\
      &   &   &   & H_{2s+1} & 0 & (-1)^2x_{p+4s}^{(l-1)} \\
      &   &   &   &   & H_{2s+1} & (-1)x_{p+2s}^{(l-1)} \\
      &   &   &   &   &   & x_p \\
  \end{array}
\right)_{(l+1)\times (l+1).}
\end{equation}
The desired result follows.

\end{proof}

To obtain our desired theorems by specific elements of Massey
products, we restrict the allowable choices of defining systems for
the two Massey products in Lemma \ref{ds} (cf. \cite{sh}). By Lemma
\ref{ds}, we give the following definitions.

\begin{definition}\label{Massey1}

(1) Given a class $[x_p]_{2t+3}\in E_{2t+3}^{p,q}$ $(t\geq 1)$, a
{specific element of $(t+2)$-fold Massey product $\langle
\underbrace{H_3,\cdots,H_3}\limits_{t+1},x_p\rangle$}, denoted by
$\langle \underbrace{H_3,\cdots,H_3}\limits_{t+1},x_p\rangle_{A}$,
is a class in $H^{p+2t+3}(M)$ represent by $c(A)$, where $A$ is a
defining system obtained from Theorems \ref{4.1}. We define the
{$(t+2)$-fold allowable Massey product $\langle
\underbrace{H_3,\cdots,H_3}\limits_{t+1},x_p\rangle_{\star}$} to be
the set of all the cohomology classes $w\in H^{p+2t+3}(M)$ for which
there exists a defining system $A$ obtained from Theorem \ref{4.1}
such that $c(A)$ represents $w$.

(2) Similarly, given a class $[x_p]_{2t+3}\in E_{2t+3}^{p,q}$
$(t\geq 1)$, when $t=ls-1$ ($l\geq 2$) we define the {specific
element of $(l+1)$-fold Massey product $\langle
\underbrace{H_{2s+1},\cdots,H_{2s+1}}\limits_{l},x_p\rangle$} and
the {$(l+1)$-fold allowable Massey product} $\langle
\underbrace{H_{2s+1},\cdots,H_{2s+1}}\limits_{l},x_p\rangle_{\star}$
by replacing Theorem \ref{4.1} by Theorem \ref{4.2} in (1).
\end{definition}

\begin{remark}\label{lxa}

(1) From Definition \ref{Massey1}, we can get the following
$$\begin{array}{l}
\langle
\underbrace{H_3,\cdots,H_3}\limits_{t+1},x_p\rangle_{\star}\subseteq
\langle \underbrace{H_3,\cdots,H_3}\limits_{t+1},x_p\rangle, \langle
\underbrace{H_3,\cdots,H_3}\limits_{t+1},x_p\rangle_{\star}\subseteq\langle
\underbrace{H_3,\cdots,H_3}\limits_{t+1},x_p\rangle.
\end{array}
$$

(2) The allowable Massey product $\langle
\underbrace{H_3,\cdots,H_3}\limits_{t+1},x_p\rangle_{\star}$ is less
ambiguous than the general Massey product $\langle
\underbrace{H_3,\cdots,H_3}\limits_{t+1},x_p\rangle$. Take $\langle
H_3,H_3,x_p \rangle_{\star}$ in Definition \ref{Massey1} for
 example. Suppose $H=\sum_{i=1}^{[\frac{n-1}{2}]}H_{2i+1}$.
By Theorem \ref{4.1} and (\ref{6.7}), there exist $x_{p+2j}^{(1)}$
such that $y_{p+2i+1}=0\quad (0\leq i \leq 1)$ and
$d_{5}[x_p]_5=[H_{3}\wedge x_{p+2}^{(1)}+H_{5}\wedge x_p]_5.$ By
Lemma \ref{ds}, we get a defining system $A$ for $\langle
H_3,H_3,x_p\rangle$ and its related cocycle $c(A)=-H_{3}\wedge
x_{p+2}^{(1)}-H_{5}\wedge x_p.$ Thus, we have
\begin{equation}\label{333}\langle H_3,H_3,x_p\rangle_{A}=[-H_{3}\wedge
x_{p+2}^{(1)}-H_{5}\wedge x_p]. \end{equation} Obviously, the
indeterminacy of the allowable Massey product $\langle
H_3,H_3,x_p\rangle_{\star}$ is $[H_3]H^{p+2}(M)$. However, in the
general case, the indeterminacy of the Massey product $\langle
H_3,H_3,x_p\rangle$ is $[H_3]H^{p+2}(M)+H^5(M)[x_p]$.

Similarly, the allowable Massey product $\langle
\underbrace{H_{2s+1},\cdots,H_{2s+1}}\limits_{l},x_p\rangle_{\star}$
is less ambiguous than the general Massey product $\langle
\underbrace{H_{2s+1},\cdots,H_{2s+1}}\limits_{l},x_p\rangle$.
\end{remark}

Now we begin to show our main theorems.

\begin{theorem}\label{differentials1}
For $H=\sum_{i=1}^{[\frac{n-1}{2}]}H_{2i+1}$ and $[x_p]_{2t+3}\in
E_{2t+3}^{p,q}$, the differential of the spectral sequence
(\ref{ss}) $d_{2t+3}: E_{2t+3}^{p, q}\to E_{2t+3}^{p+2t+3, q-2t-2}$
is given by
\begin{equation*}
d_{2t+3}[x_p]_{2t+3}=(-1)^t[\langle
\underbrace{H_3,\cdots,H_3}\limits_{t+1},x_p\rangle_{A}]_{2t+3},
\end{equation*}
and $[\langle
\underbrace{H_3,\cdots,H_3}\limits_{t+1},x_p\rangle_{A}]_{2t+3}$ is
independent of the choice of the defining system $A$ obtained from
Theorem \ref{4.1}.
\end{theorem}

\begin{proof}
By Lemma \ref{ds} (1), there exist defining systems for $\langle
\underbrace{H_3,\cdots,H_3}\limits_{t+1},x_p\rangle$ given by
Theorem \ref{4.1}. For any defining system $A=(a_{i,j})$ given by
Theorem \ref{4.1}, by (\ref{big}) we have
$$c(A)=(-1)^t(\sum\limits_{i=1}^t H_{2i+1}\wedge
x_{p+2(t-i+1)}^{(t)}+H_{2t+3}\wedge x_p).$$ By Definition
\ref{Massey1}, we have \begin{equation}\label{1234} \langle
\underbrace{H_3,\cdots,H_3}\limits_{t+1},x_p\rangle_{A}=[c(A)].
\end{equation}
Then by Theorem \ref{4.1}, we have $$\begin{array}{ll}
d_{2t+3}[x_p]_{2t+3}&=[\sum\limits_{i=1}^t H_{2i+1}\wedge
x_{p+2(t-i+1)}^{(t)}+H_{2t+3}\wedge x_p]_{2t+3}\\
&=(-1)^t [\langle
\underbrace{H_3,\cdots,H_3}\limits_{t+1},x_p\rangle_{A}]_{2t+3}.
\end{array}$$
Thus, we have $d_{2t+3}[x_p]_{2t+3}=(-1)^t[\langle
\underbrace{H_3,\cdots,H_3}\limits_{t+1},x_p\rangle_{A}]_{2t+3}.$

By the arbitrariness of $A$, we have $[\langle
\underbrace{H_3,\cdots,H_3}\limits_{t+1},x_p\rangle_{A}]_{2t+3}$ is
independent of the choice of the defining system $A$ obtained from
Theorem \ref{4.1}.
\end{proof}

\begin{example} \label{e2}
For formal manifolds which are manifolds with vanishing Massey
products, it is easy to get $$E_4^{p,q}\cong E_{\infty}^{p,q}$$ by
Theorem \ref{differentials1}. Note that simply connected compact
K\"{a}hler manifolds are an important class of formal manifolds (see
\cite{D-G-M-S}).
\end{example}

\begin{remark}\label{ljy}
(1) From the proof of the theorem above, we have that the specific
element $\langle
\underbrace{H_3,\cdots,H_3}\limits_{t+1},x_p\rangle_{A}$ represents
a class in $E_{2t+3}^{\ast,\ast}$. For two different defining
systems $A_1$ and $A_2$ given by Theorem \ref{4.1}, we have
$$\langle
\underbrace{H_3,\cdots,H_3}\limits_{t+1},x_p\rangle_{A_1}\not=\langle
\underbrace{H_3,\cdots,H_3}\limits_{t+1},x_p\rangle_{A_2}$$
generally. However, in the spectral sequence (\ref{ss}) we have
$$[\langle
\underbrace{H_3,\cdots,H_3}\limits_{t+1},x_p\rangle_{A_1}]_{2t+3}=[\langle
\underbrace{H_3,\cdots,H_3}\limits_{t+1},x_p\rangle_{A_2}]_{2t+3}.$$

(2) Since the indeterminacy of $\langle
\underbrace{H_3,\cdots,H_3}\limits_{t+1},x_p\rangle_{\star}$ does
not affect our results, we will not analyze the indeterminacy of
Massey products in this paper.

(3) By Theorem \ref{differentials1}, $d_{2t+3}[x_p]_{2t+3}=
(-1)^{t}[\langle\underbrace{H_3,\cdots,H_3}\limits_{t+1},x_p\rangle_A]_{2t+3}$
for $t\geq 1$ which is expressed only by $H_3$ and $x_p$. From the
proof of Theorem \ref{differentials1}, we know that the expression
above conceals some information, because the other $H_{2i+1}$'s
affect the result implicitly.
\end{remark}

We have the following corollary (see \cite[Proposition 6.1]{A-S}).
\begin{corollary}\label{5.1}
For $H=H_3$ only and $[x_p]_{2t+3}\in E_{2t+3}^{p,q}$ ($t\geq 1$),
we have that in the spectral sequence (\ref{ss}),
$$d_{2t+3}[x_p]_{2t+3}=
(-1)^{t}[\langle\underbrace{H_3,\cdots,H_3}\limits_{t+1},x_p\rangle_A]_{2t+3},$$
and
$[\langle\underbrace{H_3,\cdots,H_3}\limits_{t+1},x_p\rangle_A]_{2t+3}$
is independent of the choice of the defining system $A$ obtained
from Theorem \ref{4.1}.

\end{corollary}

\begin{remark}\label{5.7}\begin{enumerate}\item Because the
definition of Massey products is different from the definition in
\cite{A-S}, the expression of differentials in Corollary \ref{5.1}
differs from the one in \cite[Proposition 6.1]{A-S}.

\item The two specific elements of $\langle\underbrace{H_3,\cdots,H_3}\limits_{t+1},x_p\rangle$ in
Theorem~\ref{differentials1} and Corollary~\ref{5.1} are completely
different, and equal $[c(A_1)]$ and $[c(A_2)]$ respectively, where
$c(A_i)$ $(i=1,2)$ are related cocycles of the defining systems
$A_i$ $(i=1,2)$ obtained from Theorem \ref{4.1}. The matrices
associated to the two defining systems are given by
\begin{displaymath} \left(
  \begin{array}{ccccccc}
    H_3 & -H_5 & H_7 & \cdots & (-1)^{t-1}H_{2t+1} & (-1)^tH_{2t+3} &  \\
      & H_3 & -H_5 & \cdots & (-1)^{t-2}H_{2t-1} & (-1)^{t-1}H_{2t+1} & (-1)^tx_{p+2t}^{(t)} \\
      &   & H_3 & \cdots & (-1)^{t-3}H_{2t-3} & (-1)^{t-2}H_{2t-1} & (-1)^{t-1}x_{p+2t-2}^{(t)} \\
      &   &   & \ddots & \vdots & \vdots & \vdots \\
      &   &   &   & H_3 & -H_5 & (-1)^2x_{p+4}^{(t)} \\
      &   &   &   &   & H_3 & (-1)x_{p+2}^{(t)} \\
      &   &   &   &   &   & x_p \\
  \end{array}
\right)_{(t+2)\times (t+2)}
\end{displaymath}
and
\begin{displaymath}
\left(
  \begin{array}{ccccccc}
    H_3 & 0 & 0 & \cdots & 0 & 0 &  \\
      & H_3 & 0 & \cdots & 0 & 0 & (-1)^t{x}_{p+2t}^{(t)} \\
      &   & H_3 & \cdots & 0 & 0 & (-1)^{t-1}{x}_{p+2t-2}^{(t)} \\
      &   &   & \ddots & \vdots & \vdots & \vdots \\
      &   &   &   & H_3 & 0 & (-1)^2{x}_{p+4}^{(t)} \\
      &   &   &   &   & H_3 & (-1){x}_{p+2}^{(t)} \\
      &   &   &   &   &   & x_p \\
  \end{array}
\right)_{(t+2)\times (t+2),}
\end{displaymath}
respectively. Here $x_{p+2i}^{(t)}$ $(1\leq i\leq t)$ in the first
matrix are different from those in the second one.
\end{enumerate}\end{remark}

For $H=H_{2s+1}$ $(s\geq2)$ only (i.e., in the case $H_i=0$, $i\not=
2s+1$) and $[x_p]_{2t+3}\in E_{2t+3}^{p,q}$ $(t\geq 1)$, we make use
of Theorem \ref{differentials1} to get that
\begin{equation}\label{66}d_{2t+3}[x_p]_{2t+3}=(-1)^{t}[\langle
\underbrace{0,\cdots,0}\limits_{t+1},x_p\rangle_A]_{2t+3}.\end{equation}
Obviously, some information has been concealed in the expression
above. Now we give another description of the differentials for this
special case.

\begin{theorem}\label{differentials2}
For  $H=H_{2s+1}$ $(s\geq 1)$ only and $[x_p]_{2t+3}\in
E_{2t+3}^{p,q}$ $(t\geq 1)$, the differential of the spectral
sequence (\ref{ss}) $d_{2t+3}: E_{2t+3}^{p, q}\to E_{2t+3}^{p+2t+3,
q-2t-2}$ is given by
\[d_{2t+3}[x_p]_{2t+3}=\left\{ \begin{array}{ll}
[H_{2s+1} \wedge x_p]_{2t+3} & t=s-1,\\

(-1)^{l-1}[\langle \underbrace{H_{2s+1},\cdots,H_{2s+1}}\limits_{l},x_p\rangle_B]_{2t+3}  & t=ls-1~(l\geq 2),\\

0 & \text{otherwise,} \end{array} \right.
\]
and $[\langle
\underbrace{H_{2s+1},\cdots,H_{2s+1}}\limits_{l},x_p\rangle_{B}]_{2t+3}$
is independent of the choice of the defining system $B$ obtained
from \ref{4.2}.
\end{theorem}

\begin{proof} When $t=s-1$, the result follows from
Theorem \ref{4.2}.

When $t=ls-1$ $(l\geq 2)$, from Lemma \ref{ds} (2) we know that
there exist defining systems for $\langle
\underbrace{H_{2s+1},\cdots,H_{2s+1}}\limits_{l},x_p\rangle$
obtained from Theorem \ref{4.2}. For any defining system $B$ given
by Theorem \ref{4.2}, by (\ref{small}) we get
$c(B)=(-1)^{l-1}H_{2s+1}\wedge x_{p+2(l-1)s}^{(l-1)}.$ By Definition
\ref{Massey1},
\begin{equation}\label{0306} \langle
\underbrace{H_{2s+1},\cdots,H_{2s+1}}\limits_{l},x_p\rangle_B=[c(B)].
\end{equation}
Then by Theorem \ref{4.2}, we have $$\begin{array}{ll}
d_{2t+3}[x_p]_{2t+3}&=[H_{2s+1}\wedge x_{p+2(l-1)s}^{(l-1)}]_{2t+3}\\
&=(-1)^{l-1} [\langle
\underbrace{H_{2s+1},\cdots,H_{2s+1}}\limits_{l},x_p\rangle_B]_{2t+3}.
\end{array}$$
Thus \begin{equation*}d_{2t+3}[x_p]_{2t+3}=(-1)^{l-1} [\langle
\underbrace{H_{2s+1},\cdots,H_{2s+1}}\limits_{l},x_p\rangle_B]_{2t+3}.\end{equation*}
By the arbitrariness of $B$, we have $[\langle
\underbrace{H_{2s+1},\cdots,H_{2s+1}}\limits_{l},x_p\rangle_{B}]_{2t+3}$
is independent of the choice of the defining system $B$ obtained
from Theorem \ref{4.2}.

For the rest cases of $t$, the results follows from Theorem
\ref{4.2}.

The proof of this theorem is completed.
\end{proof}

\begin{remark}\label{5.9}
We now use the special case that $H=H_5$ and $d_9[x_p]_9$ to
illustrate the compatibility between Theorems \ref{differentials1}
and \ref{differentials2} for $s=2$ and $t=3$.

Note that in this case $H_3=0$ and $H_i=0$ for $i>5$. By Theorem
\ref{differentials1}, we get the corresponding matrix associated to
the defining system A for $\langle 0,0,0,0,x_p\rangle_{A}$ is
\begin{equation}\label{5.3}
\left(
  \begin{array}{ccccccc}
    0 & -H_5 & 0 & 0 &  \\
      & 0 & -H_5 & 0 & -x_{p+6}^{(3)} \\
      &   & 0 & -H_5 & x_{p+4}^{(3)} \\
      &   &   & 0 & -x_{p+2}^{(3)} \\
      &   &   &   & x_p \\
  \end{array}
\right)_{5\times 5}
\end{equation}and
\begin{equation}\label{xgl}\tilde{d}_9[x_p]_9=-[\langle
0,0,0,0,x_p\rangle_{A}]_9.\end{equation}

By Theorem \ref{differentials2}, in this case the matrix associated
to the defining system $B$ for $\langle H_5,H_5,x_p\rangle_{B}$ is
\begin{equation}\label{5.4}
\left(
  \begin{array}{ccccccc}
     H_5 & 0 &  \\
         & H_5 & -x_{p+4}^{(1)} \\
         &    & x_p \\
  \end{array}
\right)_{3\times 3.}
\end{equation}
and
\begin{equation}\label{lxg}\bar{d}_9[x_p]_9=-[\langle
H_5,H_5,x_p\rangle_{B}]_9.\end{equation}

We claim that $\langle H_5,H_5,x_p\rangle_{\star}=
\langle0,0,0,0,x_p\rangle_{\star}$. For any
defining system $B$ above, there is a defining system ${\tilde B}$
\begin{displaymath} \left(
  \begin{array}{ccccccc}
    0 & -H_5 & 0 & 0 &  \\
      & 0 & -H_5 & 0 & 0 \\
      &   & 0 & -H_5 & x_{p+4}^{(1)} \\
      &   &   & 0 &  0\\
      &   &   &   & x_p \\
  \end{array}
\right)_{5\times 5}
\end{displaymath} for $\langle0,0,0,0,x_p\rangle$ which can be obtained from Theorem \ref{4.1} such that
$$\langle0,0,0,0,x_p\rangle_{{\tilde B}}=\langle
H_5,H_5,x_p\rangle_B.$$ Hence $\langle
H_5,H_5,x_p\rangle_{\star}\subseteq
\langle0,0,0,0,[x_p]\rangle_{\star}.$ On the other hand, for any
defining system $A$ above there is also a defining system $\bar{A}$
\begin{equation*}
\left(
  \begin{array}{ccccccc}
     H_5 & 0 &  \\
         & H_5 & -x_{p+4}^{(3)} \\
         &    & x_p \\
  \end{array}
\right)_{3\times 3.}
\end{equation*}
for $\langle H_5,H_5,x_p\rangle$ which can be obtained from Theorem
\ref{4.2} such that
$$\langle H_5,H_5,x_p\rangle_{\bar{A}}=\langle0,0,0,0,x_p\rangle_{{A}}.$$ Therefore $\langle0,0,0,0,x_p\rangle_{\star} \subseteq\langle
H_5,H_5,x_p\rangle_{\star},$ and the claim follows.

By Theorem \ref{differentials1} and Remark \ref{lxa}, we have that
$\tilde{d}_5[y_p]_5=-[\langle 0,0,y_p\rangle_{A}]_5=-[-H_5\wedge
y_p]_5=[H_5\wedge y_p]_5$. By Theorem \ref{differentials2},
$\bar{d}_5[y_p]_5=[H_5\wedge y_p]_5$. By Proposition \ref{3.4},
$\tilde{d}_1=\bar{d}_1=d$ and $\tilde{d}_3=\bar{d}_3=0$. It follows
that
 $\tilde{d}_5=\bar{d}_5.$

By Theorems \ref{differentials1} and \ref{4.1}, we have that
$\tilde{d}_7[z_p]_7=[\langle 0,0,0,z_p\rangle_{A}]_7=[-H_5\wedge
z^{(2)}_{p+2}]_7$, where $z^{(2)}_{p+2}$ is an arbitrary
$(p+2)$-form satisfying $d(z^{(2)}_{p+2})=0\wedge z_p$. By Remark
\ref{ljy} (2), we take $z^{(2)}_{p+2}=0$. Then we have
$\tilde{d}_7[z_p]_7=0$, i.e., $\tilde{d}_7=0$. At the same time, we
also have $\bar{d}_7=0$ from Theorem \ref{differentials2}. Thus
$\tilde{d}_7=\bar{d}_7=0$.

By $\tilde{E}_1^{p,q}=\bar{E}_1^{p,q}$, $\tilde{d}_i=\bar{d}_i$ for
$1\leq i\leq 7$ and $\langle H_5,H_5,x_p\rangle_{\star}=
\langle0,0,0,0,x_p\rangle_{\star}$, we can conclude that
$\tilde{d}_9=\bar{d}_9$ from (\ref{xgl}) and (\ref{lxg}).
\end{remark}

\section{The indeterminacy of differentials in the spectral sequence (\ref{ss})}

Let $[x_p]_{r}\in E_{r}^{p, q}$. The indeterminacy of $[x_p]$ is a
normal subgroup $G$ of $H^{\ast}(M)$, which means that if there is
another element $[y_p]\in H^{p}(M)$ which also represents the class
$[x_p]_r \in E_{r}^{p, q}$, then $[y_p]-[x_p]\in
 G$.

In this section, we will show that for
$H=\sum_{i=1}^{[\frac{n-1}{2}]}H_{2i+1}$ and $[x_p]_{2t+3}$, the
indeterminacy of the differential $d_{2t+3}[x_p]\in E_{2}^{p+2t+3,
q-2t-2}$ is a normal subgroup of $H^{\ast}(M)$.

From the long exact sequence (\ref{3.3}), we have a commutative
diagram
\begin{equation}\label{6.2}
\xymatrix{& &\vdots\ar[d]_{i^*} & &\vdots \ar[d]_{i^*} & &\\
& \cdots\ar[r]^{\delta}& H_D^{p+q}(K_{p+1})\ar[d]_{i^*}
\ar[r]^{j^*}& H_D^{p+q}(K_{p+1}/K_{p+2})
\ar[r]^{\delta}&H_D^{p+q+1}(K_{p+2})\ar[d]_{i^*} \ar[r]^{j^*}&\cdots\\
&\cdots\ar[r]^{\delta}&H_D^{p+q}(K_p)\ar[d]_{i^*}\ar[r]^{j^*}&
H_D^{p+q}(K_p/K_{p+1})\ar[r]^{\delta}
&H_D^{p+q+1}(K_{p+1}) \ar[d]_{i^*}\ar[r]^{j^*}&\cdots\\
&\cdots\ar[r]^{\delta}& H_D^{p+q}(K_{p-1})\ar[d]_{i^*}\ar[r]^{j^*}&
H_D^{p+q}(K_{p-1}/K_{p})\ar[r]^{\delta}&H_D^{p+q+1}(K_{p}) \ar[d]_{i^*}\ar[r]^{j^*}&\cdots\\
& &\vdots  & &\vdots  & &\\}
\end{equation}
in which any sequence consisting of a vertical map $i^{*}$ followed
by two horizontal maps $j^*$ and $\delta$ and then a vertical map
$i^*$ followed again by $j^{\ast}$ and $\delta$ and iteration of
this is exact. From this diagram there is obtained a spectral
sequence in which $E_1^{p,q}=H_D^{p+q}(K_p/K_{p+1})$ and for $r\geq
2$, $E_r^{p,q}$ is defined to be the quotient
$Z_{r}^{p,q}/B_r^{p,q}$, where
\begin{equation}\label{6.00}
\begin{array}{l}
Z_r^{p,q}=\delta^{-1}(i^{*r-1}H_D^{p+q+1}(K_{p+r})),\\
B_r^{p,q}=j^*({\rm ker}[i^{*r-1}: H_D^{p+q}(K_p)\rightarrow H_D^{p+q}(K_{p-r+1})]).\\
\end{array}
\end{equation}
We also have a sequence of inclusions
\begin{equation}\label{6.10}
B_2^{p,q}\subset\cdots\subset B_r^{p,q}\subset
B_{r+1}^{p,q}\subset\cdots\subset Z_{r+1}^{p,q}\subset
Z_r^{p,q}\subset\cdots\subset Z_2^{p,q}.
\end{equation}
By \cite{Ma,Ma1}, the $E_r^{\ast,\ast}$-term defined above is the
same as the one in the spectral sequence (\ref{ss}). A similar
argument about a homology spectral sequence is given in \cite[p.
472-473]{S}.

\begin{theorem}\label{6.1}
Let $H=\sum_{i=1}^{[\frac{n-1}{2}]}H_{2i+1}$ and $[x_p]_r\in
E_{r}^{p, q}$ $(r\geq 3)$, then the indeterminacy of $[x_p]\in
E_{2}^{p, q}\cong H^{p}(M)$ is the following normal subgroup of
$H^{p}(M)$£º
\begin{displaymath}
\frac{{\rm im} [\bar{\delta}: H_D^{p+q-1}(K_{p-r+1}/K_p)\to
H_D^{p+q}(K_{p}/K_{p+1})]}{{\rm im} [d:\Omega^{p-1}(M)\to
\Omega^{p}(M)]},
\end{displaymath}
where $d$ is just the exterior differentiation and $\bar{\delta}$ is
the connecting homomorphism of the long exact sequence induced by
the short exact sequence of cochain complexes
$$0\longrightarrow K_p/K_{p+1}\stackrel{\bar{i}}\longrightarrow
K_{p-r+1}/K_{p+1}\stackrel{\bar{j}}\longrightarrow
K_{p-r+1}/K_{p}\longrightarrow 0.$$
\end{theorem}

\begin{proof}

From the tower (\ref{6.10}) above, we get a tower of subgroups of
$E_2^{p,q}$
\begin{equation*}\begin{split}
B_3^{p,q}/B_2^{p,q}&\subset\cdots\subset B_r^{p,q}/B_2^{p,q}
\subset\cdots \subset
Z_r^{p,q}/B_2^{p,q}\subset\\
&\cdots \subset Z_3^{p,q}/B_2^{p,q}\subset Z_2^{p,q}/B_2^{p,q}=
E_2^{p,q}.
\end{split}
\end{equation*}
Note that $E_{r}^{p,q}\cong
(Z_{r}^{p,q}/B_2^{p,q})/(B_{r}^{p,q}/B_2^{p,q}).$ It follows that
the indeterminacy of $[x_p]$ is the normal subgroup
$B_{r}^{p,q}/B_2^{p,q}$ of $H^{p}(M)$.

From the short exact sequences of cochain complexes
$$\begin{array}{c}
0\longrightarrow K_{p}\stackrel{i^{\prime}}\longrightarrow
K_{p-r+1}\stackrel{j^{\prime}}\longrightarrow
K_{p-r+1}/K_{p}\longrightarrow 0,\\
0\longrightarrow K_p/K_{p+1}\stackrel{\bar{i}}\longrightarrow
K_{p-r+1}/K_{p+1}\stackrel{\bar{j}}\longrightarrow
K_{p-r+1}/K_{p}\longrightarrow 0,
\end{array}
$$
we can get the following long exact sequence of cohomology groups
\begin{equation}\label{9.0}\begin{array}{c}
\cdots\stackrel{\delta^{\prime}}{\longrightarrow}
H_D^{s}(K_{p})\stackrel{i^{\prime\ast}}\longrightarrow
H_D^{s}(K_{p-r+1})\stackrel{j^{\prime\ast}}\longrightarrow
H_D^{s}(K_{p-r+1}/K_{p})\stackrel{\delta^{\prime}}\longrightarrow
\cdots,\\
\cdots\stackrel{\bar{\delta}}{\longrightarrow}
H_D^{s}(K_{p}/K_{p+1})\stackrel{\bar{i}^{\ast}}\longrightarrow
H_D^{s}(K_{p-r+1}/K_{p+1})\stackrel{\bar{j}^{\ast}}\longrightarrow
H_D^{s}(K_{p-r+1}/K_{p})\stackrel{\bar{\delta}}\longrightarrow
\cdots,
\end{array}
\end{equation}
where $\delta^{\prime}$ and $\bar{\delta}$ are the connecting
homomorphisms.

Combining (\ref{3.3}) and (\ref{9.0}), we have the following
commutative diagram of long exact sequences
\begin{equation}\label{100}
{\xymatrix{
  H_D^{p+q-1}(K_{p-r+1}/K_p) \ar[dr]_{\bar{\delta}} \ar[r]^{\delta^{\prime}} & H_D^{p+q}(K_p) \ar[d]^{j^*} \ar[r]^{i^{{\prime\ast}}} &H_D^{p+q}(K_{p-r+1}) \\
               &  H_D^{p+q}(K_{p}/K_{p+1}) \ar[d]^{\delta} \ar[dr]_{\bar{i}^*}  &   \\
                & H_D^{p+q+1}(K_{p+1}) & H_D^{p+q}(K_{p-r+1}/K_{p+1}).
                }}
\end{equation}

Using the commutative diagram above and the fact that $i^{{\ast}^
{r-1}}={i}^{\prime\ast}$, we have
\begin{displaymath}
\begin{array}{ll} B_{r}^{p,q}&=j^*({\rm
ker}[i^{*^{r-1}}: H_D^{p+q}(K_p)\rightarrow
H_D^{p+q}(K_{p-r+1})])\\&=j^*({\rm ker}[~
i^{\prime\ast}:H_D^{p+q}(K_p)\longrightarrow H_D^{p+q}(K_{p-r+1})])\\
&\cong j^*({\rm im}[{\delta}^{'}: H_D^{p+q-1}(K_{p-r+1}/K_p)\rightarrow H_D^{p+q}(K_p)])\\
&\cong{\rm im} [\bar{\delta}: H_D^{p+q-1}(K_{p-r+1}/K_p)\to
H_D^{p+q}(K_{p}/K_{p+1})].\end{array}
\end{displaymath}

When $r=2$, from (\ref{100}) we have that
$$\bar{\delta}=\delta^{\prime}j^{\ast}: H_D^{p+q-1}(K_{p-1}/K_p)\to
H_D^{p+q}(K_{p}/K_{p+1}).$$ From (\ref{3.4}), it follows that
$\bar{\delta}=d_1$. By Proposition \ref{e3}, $\bar{\delta}=d$. Thus
we have
$$\begin{array}{ll}
B_2^{p,q}&\cong {\rm im} [\bar{\delta}: H_D^{p+q-1}(K_{p-1}/K_p)\to
H_D^{p+q}(K_{p}/K_{p+1})]\\
&\cong {\rm im} [d:\Omega^{p-1}(M)\to \Omega^{p}(M)].\end{array}$$

The desired result follows.
\end{proof}

By Theorem \ref{6.1}, we obtain the following corollary.

\begin{corollary}
In Theorem \ref{differentials1}, for $d_{2t+3}[x_p]_{2t+3}\in
E_{2t+3}^{p+2t+3, q-2t-2}$ we have the indeterminacy of
$d_{2t+3}[x_p]$ is a normal subgroup of $H^{p+2t+3}(M)$
\begin{displaymath}
\frac{{\rm im} [\bar{\delta}: H_D^{p+q}(K_{p+1}/K_{p+2t+3})\to
H_D^{p+q+1}(K_{p+2t+3}/K_{p+2t+4})]}{{\rm im} [d:
\Omega^{p+2t+2}(M)\to \Omega^{p+2t+3}(M)]},
\end{displaymath}
where $d$ is just the exterior differentiation and $\bar{\delta}$ is
the connecting homomorphism of the long exact sequence induced by
the short exact sequence of cochain complexes
$$0\longrightarrow K_{p+2t+3}/K_{p+2t+4}\stackrel{\bar{i}}\longrightarrow
K_{p+1}/K_{p+2t+4}\stackrel{\bar{j}}\longrightarrow
K_{p+1}/K_{p+2t+3}\longrightarrow 0.$$
\end{corollary}

\begin{proof}
In Theorem \ref{6.1} $r, p$ and $q$ are replaced by $2t+3, p+2t+3$
and $q-2t-2$, then the desired result follows.
\end{proof}

\noindent{\bf Acknowledgment} The authors would like to thank Jim
Stasheff for helpful comments.

\end{document}